\documentclass{amsart}
\usepackage{amsmath,amssymb,amsthm,amscd,latexsym}
\usepackage[dvipdfmx]{graphicx}
\usepackage[all]{xy}
\usepackage{wrapfig}
\usepackage{color}
\usepackage{enumerate}
\usepackage{cases}
\usepackage[top=30mm,bottom=30mm,left=30mm,right=30mm]{geometry}
\usepackage{graphics}
\usepackage{tikz}
\usetikzlibrary{patterns}
\usepackage{pxpgfmark}
\usepackage{multirow}
\usepackage{extarrows}
\usepackage{comment}
\usetikzlibrary{intersections, calc,decorations.markings}

\numberwithin{equation}{section}
\theoremstyle{definition}
\newtheorem{example}{Example}[section]
\newtheorem{definition}[example]{Definition}

\theoremstyle{plain}
\newtheorem{lemma}[example]{Lemma}
\newtheorem{theorem}[example]{Theorem}
\newtheorem{proposition}[example]{Proposition}
\newtheorem{corollary}[example]{Corollary}

\newlength\squareheight
\DeclareMathOperator{\bA}{\mathbb{A}}
\DeclareMathOperator{\bQ}{\mathbb{Q}}
\DeclareMathOperator{\bR}{\mathbb{R}}
\DeclareMathOperator{\bT}{\mathbb{T}}
\DeclareMathOperator{\bZ}{\mathbb{Z}}

\DeclareMathOperator{\cA}{\mathcal{A}}

\DeclareMathOperator{\cF}{\mathcal{F}}
\DeclareMathOperator{\cO}{\mathcal{O}}

\DeclareMathOperator{\be}{\mathbf{e}}
\DeclareMathOperator{\bx}{\mathbf{x}}

\DeclareMathOperator{\se}{\mathsf{e}}

\DeclareMathOperator{\sT}{\mathsf{T}}

\DeclareMathOperator{\clv}{\operatorname{Cl-var}}
\DeclareMathOperator{\cluster}{\operatorname{Cluster}}

\title{Denseness of $g$-vector cones from weighted orbifolds}
\author{Toshiya Yurikusa}
\address{Laboratoire de Math\'{e}matiques de Versailles, UVSQ, CNRS, Universit\'{e} Paris-Saclay, 78035 Versailles, France and Mathematical Institute, Tohoku University, Sendai 980-8578, Japan}
\email{toshiya.yurikusa.d8@tohoku.ac.jp}

\begin{document}

\begin{abstract}
We study $g$-vector cones in a cluster algebra defined from a weighted orbifold of rank $n$ introduced by Felikson, Shapiro and Tumarkin. We determine the closure of the union of the $g$-vector cones. It is equal to $\mathbb{R}^n$ except for a weighted orbifold with empty boundary and exactly one puncture, in which case it is equal to the half space of a certain explicit hyperplane in $\mathbb{R}^n$.
\end{abstract}
\maketitle
%%%
%%%
%%%
\section{Introduction}

Cluster algebras \cite{FZ02} are commutative algebras with generators, called cluster variables, which are grouped into sets of fixed cardinality, called clusters. Their original motivation was to study total positivity of semisimple Lie groups and canonical bases of quantum groups.  In recent years, cluster algebras have interacted with various subjects in mathematics, for example, representation theory of quivers, Poisson geometry, integrable systems, and so on.

Let $B$ be a skew-symmetrizable matrix and $\cA(B)$ the associated cluster algebra with principal coefficients (see Section \ref{sec:cluster algebra}). We denote by $\cluster B$ the set of clusters of $\cA(B)$. Each cluster variable $x$ of $\cA(B)$ has a numerical invariant $g(x)$, called the $g$-vector of $x$ \cite{FZ07}. For each $\bx\in\cluster B$, one can define a cone
\[
C(\bx) := \left\{\sum_{x \in \bx} a_x g(x) \middle| a_x \in \bR_{\ge0}\right\}
\]
in $\bR^n$, called the $g$-vector cone of $\bx$. All $g$-vector cones in $\cA(B)$ and their faces form a fan \cite[Theorem 8.7]{Re14}, called the $g$-vector fan of $\cA(B)$. It appears in various subjects: For example, it is a subfan of the underlying fan of the cluster scattering diagram \cite{GHKK18} and of the stability scattering diagram \cite{B17}; For cluster algebras of finite type, it is a complete fan \cite{FZ03b} and the normal fan of the generalized associahedron \cite{BCDMTY18,CFZ02,HPS18,PPPP19}. Moreover, the notion of $g$-vector fans is also defined in representation theory of quivers via $\tau$-tilting or silting theory, and studied in many papers (e.g. \cite{AHIKM22,AHIKM23,A21,BST19,DIJ19,M22}).

We study the denseness of $g$-vector fans, that is, of the union of $g$-vector cones. It is known for the following classes (see also \cite{Y23}):
\begin{itemize}
\item Cluster algebras of finite type \cite{FZ03b} or affine type \cite{RS18};
\item Cluster algebras and Jacobian algebras defined from marked surfaces \cite{Y20};
\item $\tau$-tilting finite algebras \cite{DIJ19};
\item Finite dimensional tame algebras \cite{PY23};
\item Complete special biserial algebras \cite{AY23};
\item Complete preprojective algebras of extended Dynkin quivers \cite{KM22}.
\end{itemize}
Note that the denseness can be used to study the connectedness of exchange graphs in representation theory \cite[Corollary 1.4]{Y20} and the full Fock-Goncharov conjecture \cite[Proposition 5.22 and Theorem 5.24]{MQ23} and so on. In this paper, we add cluster algebras defined from weighted orbifolds to these classes.

Let $\cO^w$ be a weighted orbifold and $T$ a tagged triangulation of $\cO^w$ (see Section \ref{sec:wo}). Felikson, Shapiro and Tumarkin \cite{FeST12a,FeST12b} constructed the skew-symmetrizable matrix $B_T$ associated with $T$ and studied the corresponding cluster algebra $\cA(B_T)$. Cluster variables and clusters of $\cA(B_T)$ correspond to tagged arcs and tagged triangulations of $\cO^w$, respectively (Theorem \ref{thm:bijection}). Note that a weighted orbifold without orbifold points is just a marked surface that was developed in \cite{FG06,FG09,FoST08,FT18,GSV05}. Our main aim is to give the following result proven in \cite[Theorem 1.2]{Y20} for marked surfaces.

%%%
\begin{theorem}\label{thm:dense wo}
If $\cO^w$ is not a weighted orbifold with empty boundary and exactly one puncture, then we have
\[
\overline{\bigcup_{\bx\in\cluster B_T}C(\bx)}=\bR^{|T|}.
\]
If $\cO^w$ is a weighted orbifold with empty boundary and exactly one puncture, then we have
\[
\overline{\bigcup_{\bx\in\cluster B_T}C(\bx)}=
\left\{(a_{\delta})_{\delta\in T}\in\bR^{|T|}\middle|\sum_{\delta\in T\setminus T_p}a_{\delta}+\frac{1}{2}\sum_{\delta\in T_p}a_{\delta}\ge 0\right\},
\]
where $T_p$ is the set of pending arcs of $T$ with weight $\frac{1}{2}$.
\end{theorem}

We prove Theorem \ref{thm:dense wo} by using the method in \cite{Y20}. In fact, we consider a certain class of curves in $\cO^w$, called laminates (see Subsection \ref{subsec:lam}). To a laminate $\ell$ of $\cO^w$, we associate an integer vector $b_T(\ell) \in \bZ^{|T|}$ whose entries are shear coordinates of $\ell$.  In the same way as $g$-vector cones, we can define a cone $C(L)$ in $\bR^{|T|}$ for a set $L$ of laminates, that is, $C(L):=\{\sum_{\ell\in L}a_{\ell}b_T(\ell)\mid a_{\ell}\in\bR_{\ge0}\}$. One can construct an injective map $\se$ from the set of tagged arcs to the set of laminates (see Subsection \ref{subsec:kinds of lam}). The following result plays an important role to prove Theorem \ref{thm:dense wo}.

%%%
\begin{theorem}\label{thm:dense lam}
For any weight orbifold $\cO^w$, we have
\[
\overline{\bigcup_{T'}C(\se(T'))}=\bR^{|T|},
\]
where $T'$ runs over all tagged triangulations of $\cO^w$. In addition, if $\cO^w$ is a weighted orbifold with empty boundary and exactly one puncture, then we have
\[
\overline{\bigcup_{T'}C(\se(T'))} = \overline{\bigcup_{T''}(-C(\se(T'')))}=
\left\{(a_{\delta})_{\delta\in T}\in\bR^{|T|}\middle|\sum_{\delta\in T\setminus T_p}a_{\delta}+\frac{1}{2}\sum_{\delta\in T_p}a_{\delta}\le 0\right\},
\]
where $T'$ (resp., $T''$) runs over all tagged triangulations of $\cO^w$ tagged at the puncture in the same (resp., different) way as $T$.
\end{theorem}

This paper is organized as follows. In Section \ref{sec:wo}, we recall the notions of weighted orbifolds, laminates, and their shear coordinates.  We classify laminates into four kinds: elementary, exceptional, semi-closed, and closed laminates. By replacing exceptional and semi-closed laminates with appropriate elementary and closed laminates, respectively, we can reduce to two kinds of laminates to prove Theorem \ref{thm:dense lam}. We apply the Dehn twists along closed laminates to elementary laminates. Seeing the asymptotic behavior of their shear coordinates, we prove Theorem \ref{thm:dense lam}. In Section \ref{sec:cluster algebra}, we recall cluster algebras defined from weighted orbifolds and explain that Theorem \ref{thm:dense wo} follows from Theorem \ref{thm:dense lam}.

This paper is a natural extension of \cite{Y20}, which is the author's previous paper for marked surfaces. For that reason, we need almost similar claims and proofs obtained by just replacing marked surfaces with weighted orbifolds (see e.g. Propositions \ref{prop:properties of elementary}, \ref{prop:properties of pq}, and \ref{prop:inclusion}). We refer to the corresponding results in \cite{Y20} before the claims and state the proofs for the convenience of the reader.

%%%
%%%
%%% sec:wo
\section{Denseness of cones from tagged triangulations on a weighted orbifold}\label{sec:wo}

%%%
%%%
%%%
\subsection{Orbifolds}

We start with recalling the notions of \cite{FeST12a,FoST08}. Let $S$ be a connected compact oriented Riemann surface with (possibly empty) boundary $\partial S$. Let $M \sqcup Q$ be a finite set of marked points on $S$ such that $M\neq\emptyset$, $Q\cap\partial S=\emptyset$, and there are at least one marked point on each connected component of $\partial S$. We call the triple $\cO=(S,M,Q)$ an \emph{orbifold}. Note that the orbifold $(S,M,\emptyset)$ without orbifold points is also called a \emph{marked surface}. A marked point in $M\setminus\partial S$ (resp., in $Q$) is called a \emph{puncture} (resp., an \emph{orbifold point}). We represent an orbifold point by $\times$ in the figures. For technical reasons, we assume that $\cO$ is none of the following (see \cite{FoST08} for the details):
\begin{itemize}
\item a monogon with at most one puncture and orbifold point in total;
\item a digon without punctures nor orbifold points;
\item a triangle without punctures nor orbifold points;
\item a sphere with at most three punctures and orbifold points in total.
\end{itemize}
 Throughout this paper, a curve in $S$ is considered up to isotopy relative to $M\sqcup Q$. When we consider intersections of curves, we assume that they intersect transversally in a minimum number of points.

%%% ideal arc
An \emph{ideal arc} $\gamma$ of $\cO$ is a curve in $S$ with endpoints in $M \sqcup Q$ such that the following conditions are satisfied (see Figure \ref{fig:nonarcs}):
\begin{itemize}
\item $\gamma$ does not intersect itself except at its endpoints;
\item $\gamma$ is disjoint from $M\cup Q\cup\partial S$ except at its endpoints;
\item $\gamma$ does not connect orbifold points;
\item If $\gamma$ cut out a monogon, then its interior has either at least one puncture or at least two orbifold points;
\item $\gamma$ does not cut out a digon without punctures nor orbifold points.
\end{itemize}
%%%
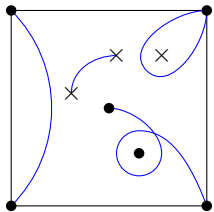
\begin{figure}[htp]
\begin{tikzpicture}[baseline=0mm]
\coordinate(lu)at(-1.3,1.3);\coordinate(ld)at(-1.3,-1.3);\coordinate(ru)at(1.3,1.3);\coordinate(rd)at(1.3,-1.3);
\draw(lu)--(ld)--(rd)--(ru)--(lu);
\coordinate(o1)at(0.1,0.7);\coordinate(o2)at(-0.5,0.2);\coordinate(o3)at(0.7,0.7);
\coordinate(p)at(0,0);\coordinate(q)at(0.4,-0.6);
\draw[blue](o1)to[out=180,in=90](o2);
\draw[blue](ru)to[out=45,in=90,relative](0.5,0.5);
\draw[blue](ru)to[out=-45,in=-90,relative](0.5,0.5); 
\draw[blue](lu)to[out=-45,in=45](ld);
\draw[blue](rd)..controls(1,-0.5)and(0.7,-0.3)..(0.4,-0.3); 
\draw[blue](0.4,-0.3)arc(90:360:0.3); 
\draw[blue](p)..controls(0.3,0)and(0.7,-0.3)..(0.7,-0.6);
\fill(lu)circle(0.07);\fill(ld)circle(0.07);\fill(ru)circle(0.07);\fill(rd)circle(0.07);
\fill(p)circle(0.07);\fill(q)circle(0.07);
\node at(o1){$\times$};\node at(o2){$\times$};\node at(o3){$\times$};
\end{tikzpicture}
\caption{Curves in an orbifold which are not ideal arcs}
\label{fig:nonarcs}
\end{figure}
%%%

%%% ideal triangulation
We say that two ideal arcs are \emph{compatible} if they do not intersect in the interior of $S$ and they are not incident to a common orbifold point. An \emph{ideal triangulation} is a maximal set of distinct pairwise compatible ideal arcs.
%%% self-folded
A triangle with only two distinct sides is called \emph{self-folded} (see Figure \ref{fig:self-folded}).
\begin{figure}[htp]
\begin{minipage}{0.5\textwidth}
\centering
\begin{tikzpicture}
\coordinate(d)at(0,0);\coordinate(p)at(0,1);
\draw(d)node[below]{$o$}--node[fill=white,inner sep=1]{$\gamma'$}(p)node[above]{$p$};
\draw(d)to[out=30,in=0]node[right]{$\gamma$}(0,1.6);\draw(d)to[out=150,in=180](0,1.6);
\fill(d)circle(0.07);\fill(p)circle(0.07);
\end{tikzpicture}
\hspace{10mm}
\begin{tikzpicture}
\coordinate(d)at(0,0);\coordinate(p)at(0,1);
\draw(d)node[below]{$o$}--node[left]{$\iota(\gamma')$}(p)node[above]{$p$};
\draw(d)to[out=10,in=-10]node[right]{$\iota(\gamma)$}node[pos=0.8]{\rotatebox{40}{\footnotesize $\bowtie$}}(p);
\fill(d)circle(0.07);\fill(p)circle(0.07);
\end{tikzpicture}
\caption{A self-folded triangle and the corresponding tagged arcs}
\label{fig:self-folded}
\end{minipage}
%%%
\begin{minipage}{0.48\textwidth}\vspace{5mm}
\centering
\begin{tikzpicture}
\coordinate(0)at(0,0);\coordinate(1)at(0,-1.2);
\draw(0)--(1);
\draw(0)to[out=80,in=100,relative]node[pos=0.2]{\rotatebox{40}{\footnotesize $\bowtie$}}(1);
\fill(0)circle(0.07);\fill(1)circle(0.07);
\end{tikzpicture}
\hspace{10mm}
\begin{tikzpicture}
\coordinate(0)at(0,0);\coordinate(1)at(0,-1.2);
\draw(0)--node[pos=0.8]{\rotatebox{0}{\footnotesize $\bowtie$}}(1);
\draw(0)to[out=80,in=100,relative]node[pos=0.2]{\rotatebox{40}{\footnotesize $\bowtie$}}node[pos=0.75]{\rotatebox{-40}{\footnotesize $\bowtie$}}(1);
\fill(0)circle(0.07);\fill(1)circle(0.07);
\end{tikzpicture}
\caption{Pairs of conjugate arcs}
\label{fig:pair}
\end{minipage}
\end{figure}
%%%

%%% flip
For an ideal triangulation $T$, a \emph{flip} at an ideal arc $\gamma \in T$ replaces $\gamma$ with another ideal arc $\gamma' \notin T$ such that $(T\setminus\{\gamma\})\cup\{\gamma'\}$ is an ideal triangulation. Notice that an ideal arc inside a self-folded triangle can not be flipped. To make flip always possible, the notion of tagged arcs was introduced in \cite{FoST08}.

%%% tagged arc
A \emph{tagged arc} $\delta$ of $\cO$ is an ideal arc with each end being tagged in one of two ways, \emph{plain} or \emph{notched}, such that the following conditions are satisfied (see Figure \ref{fig:nontaggedarcs}):
\begin{itemize}
 \item $\delta$ does not cut out a monogon with exactly one puncture and no orbifold points;
 \item Ends of $\delta$ incident to $Q\sqcup\partial S$ are tagged plain;
 \item Both ends of a loop are tagged in the same way,
\end{itemize}
where a \emph{loop} is an ideal/tagged arc with two identical endpoints. Note that the endpoints of a loop are not orbifold points. An ideal/tagged arc with endpoint being an orbifold point is called a \emph{pending} (\emph{ideal/tagged}) \emph{arc}. In the figures, we represent tags as follows:
%%% tagged %%%
\[
\begin{tikzpicture}
\coordinate(0)at(0,0) node[left]{plain};
\coordinate(1)at(1,0);  \fill(1)circle(0.07);
\draw(0) to (1);
\end{tikzpicture}
\hspace{7mm}
\begin{tikzpicture}
\coordinate(0)at(0,0) node[left]{notched};
\coordinate(1)at(1,0);  \fill(1)circle(0.07);
\draw(0) to node[pos=0.8]{\rotatebox{90}{\footnotesize $\bowtie$}} (1);
\end{tikzpicture}
\]
%%%
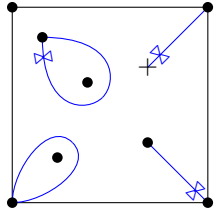
\begin{figure}[htp]
\begin{tikzpicture}[baseline=0mm]
\coordinate(lu)at(-1.3,1.3);\coordinate(ld)at(-1.3,-1.3);\coordinate(ru)at(1.3,1.3);\coordinate(rd)at(1.3,-1.3);
\draw(lu)--(ld)--(rd)--(ru)--(lu);
\coordinate(p1)at(-0.9,0.9);\coordinate(q1)at(-0.3,0.3);\coordinate(o)at(0.5,0.5);
\coordinate(p)at(-0.7,-0.7);\coordinate(q)at(0.5,-0.5);
\draw[blue](ru)--node[pos=0.8]{\rotatebox{-45}{\footnotesize $\bowtie$}}(o);
\draw[blue](rd)--node[pos=0.2]{\rotatebox{45}{\footnotesize $\bowtie$}}(q);
\draw[blue](p1)..controls(0.1,0.9)and(0.1,0.3)..(-0.1,0.1); 
\draw[blue](p1)..controls(-0.9,-0.1)and(-0.3,-0.1)..node[pos=0.1]{\rotatebox{10}{\footnotesize $\bowtie$}}(-0.1,0.1);
\draw[blue](ld)to[out=45,in=90,relative](-0.5,-0.5);
\draw[blue](ld)to[out=-45,in=-90,relative](-0.5,-0.5); 
\fill(lu)circle(0.07);\fill(ld)circle(0.07);\fill(ru)circle(0.07);\fill(rd)circle(0.07);
\fill(p)circle(0.07);\fill(q)circle(0.07);\fill(p1)circle(0.07);\fill(q1)circle(0.07);
\node at(o){\rotatebox{45}{$\times$}};
\end{tikzpicture}
\caption{Ideal arcs with each end being tagged, which are not tagged arcs}
\label{fig:nontaggedarcs}
\end{figure}
%%%

%%% ideal arc and tagged arc
For an ideal arc $\gamma$ of $\cO$, we define a tagged arc $\iota(\gamma)$ as follows:
\begin{itemize}
\item If $\gamma$ does not cut out a monogon with exactly one puncture and no orbifold points, then $\iota(\gamma)$ is the tagged arc obtained from $\gamma$ by tagging both ends plain;
\item If $\gamma$ cuts out a monogon with exactly one puncture $p$ and no orbifold points, then there is a unique ideal arc $\gamma'$ such that $\{\gamma,\gamma'\}$ is a self-folded triangle. Then $\iota(\gamma)$ is the tagged arc obtained from $\iota(\gamma')$ by changing its tag at $p$ (see Figure \ref{fig:self-folded}).
\end{itemize}
A \emph{pair of conjugate arcs} is, for a self-folded triangle $\{\gamma,\gamma'\}$, either $\{\iota(\gamma),\iota(\gamma')\}$ or obtained from it by changing all tags at each endpoint (see Figure \ref{fig:pair}).

%%% tagged triangulation
For a tagged arc $\delta$, we denote by $\delta^{\circ}$ the ideal arc obtained from $\delta$ by forgetting its tags. We say that two tagged arcs $\delta$ and $\varepsilon$ are \emph{compatible} if the following conditions are satisfied:
\begin{itemize}
\item The ideal arcs $\delta^{\circ}$ and $\varepsilon^{\circ}$ are compatible;
\item If $\delta^{\circ}=\varepsilon^{\circ}$, then $\delta=\varepsilon$ or $\{\delta,\varepsilon\}$ is a pair of conjugate arcs;
\item If $\delta^{\circ}\neq\varepsilon^{\circ}$ and they have a common endpoint $o$, then the tags of $\delta$ and $\varepsilon$ at $o$ are the same.
\end{itemize}
A \emph{partial tagged triangulation} is a set of distinct pairwise compatible tagged arcs. If a partial tagged triangulation is maximal, then it is called a \emph{tagged triangulation}. It is easy to check that any tagged triangulation of $\cO$ decomposes $\cO$ into \emph{triangles} as in Figures \ref{Fig:non-exceptional-triangles} and \ref{Fig:exceptional-triangles} (see also \cite{FeST12a,FeST12b}), which are called non-exceptional triangles and exceptional triangles, respectively.

%%%
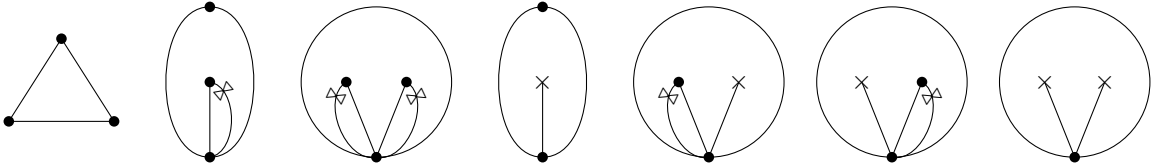
\begin{figure}[htp]
\centering
\begin{tikzpicture}[baseline=12]
\coordinate(u)at(0,1);\coordinate(l)at(-0.7,-0.1);\coordinate(r)at(0.7,-0.1);
\draw(u)--(l)--(r)--(u);
\fill(u)circle(0.07);\fill(l)circle(0.07);\fill(r)circle(0.07);
\end{tikzpicture}
 \hspace{3mm}%
\begin{tikzpicture}[baseline=0]
\coordinate(p)at(0,0);\coordinate(u)at(0,1);\coordinate(d)at(0,-1);
\draw(d)to[out=180,in=180](u);\draw(d)to[out=0,in=0](u);
\draw(d)--(p);
\draw(p)to[out=80,in=100,relative]node[pos=0.2]{\rotatebox{40}{\footnotesize $\bowtie$}}(d);
\fill(p)circle(0.07);\fill(u)circle(0.07);\fill(d)circle(0.07);
\end{tikzpicture}
 \hspace{3mm}%
\begin{tikzpicture}[baseline=0]
\coordinate(d)at(0,-1);\coordinate(l)at(-0.4,0);\coordinate(r)at(0.4,0);
\draw(0,0)circle(1);\draw(l)--(d)--(r);
\draw(l)to[out=-80,in=-100,relative]node[pos=0.2]{\rotatebox{-25}{\footnotesize $\bowtie$}}(d);
\draw(r)to[out=80,in=100,relative]node[pos=0.2]{\rotatebox{25}{\footnotesize $\bowtie$}}(d);
\fill(d)circle(0.07);\fill(l)circle(0.07);\fill(r)circle(0.07);
\end{tikzpicture}
 \hspace{3mm}%
\begin{tikzpicture}[baseline=0]
\coordinate(p)at(0,0);\coordinate(u)at(0,1);\coordinate(d)at(0,-1);
\draw(d)to[out=180,in=180](u);\draw(d)to[out=0,in=0](u);
\draw(d)--(p);
\node at(p){$\times$};\fill(u)circle(0.07);\fill(d)circle(0.07);
\end{tikzpicture}
 \hspace{3mm}%
\begin{tikzpicture}[baseline=0]
\coordinate(d)at(0,-1);\coordinate(l)at(-0.4,0);\coordinate(r)at(0.4,0);
\draw(0,0)circle(1);\draw(l)--(d)--(r);
\draw(l)to[out=-80,in=-100,relative]node[pos=0.2]{\rotatebox{-25}{\footnotesize $\bowtie$}}(d);
\fill(d)circle(0.07);\fill(l)circle(0.07);\node at(r){$\times$};
\end{tikzpicture}
 \hspace{3mm}%
\begin{tikzpicture}[baseline=0]
\coordinate(d)at(0,-1);\coordinate(l)at(-0.4,0);\coordinate(r)at(0.4,0);
\draw(0,0)circle(1);\draw(l)--(d)--(r);
\draw(r)to[out=80,in=100,relative]node[pos=0.2]{\rotatebox{25}{\footnotesize $\bowtie$}}(d);
\fill(d)circle(0.07);\node at(l){$\times$};\fill(r)circle(0.07);
\end{tikzpicture}
 \hspace{3mm}%
\begin{tikzpicture}[baseline=0]
\coordinate(d)at(0,-1);\coordinate(l)at(-0.4,0);\coordinate(r)at(0.4,0);
\draw(0,0)circle(1);\draw(l)--(d)--(r);
\fill(d)circle(0.07);\node at(l){$\times$};\node at(r){$\times$};
\end{tikzpicture}
\caption{A complete list of non-exceptional triangles of orbifolds except for changing all tags at each puncture}
\label{Fig:non-exceptional-triangles}
\end{figure}
%%%
%%% exceptional triangles
\begin{figure}[htp]
\centering
\begin{tikzpicture}[baseline=0mm]
\coordinate(0)at(0,0);\coordinate(u)at(90:1);\coordinate(r)at(-30:1);\coordinate(l)at(210:1);
\draw(0)--(u);\draw(0)--(l);\draw(0)--(r);
\draw(0)to[out=-60,in=-120,relative]node[pos=0.8]{\rotatebox{20}{\footnotesize $\bowtie$}}(u);
\draw(0)to[out=-60,in=-120,relative]node[pos=0.8]{\rotatebox{-30}{\footnotesize $\bowtie$}}(l);
\draw(0)to[out=-60,in=-120,relative]node[pos=0.8]{\rotatebox{100}{\footnotesize $\bowtie$}}(r);
\fill(0)circle(0.07);\fill(u)circle(0.07);\fill(l)circle(0.07);\fill(r)circle(0.07);\node at(0.3,0){$o$};
\end{tikzpicture}
 \hspace{7mm}%
\begin{tikzpicture}[baseline=0mm]
\coordinate(0)at(0,0);\coordinate(u)at(90:1);\coordinate(r)at(-30:1);\coordinate(l)at(210:1);
\draw(0)--(u);\draw(0)--(l);\draw(0)--(r);
\draw(0)to[out=-60,in=-120,relative]node[pos=0.8]{\rotatebox{20}{\footnotesize $\bowtie$}}(u);
\draw(0)to[out=-60,in=-120,relative]node[pos=0.8]{\rotatebox{-30}{\footnotesize $\bowtie$}}(l);
\fill(0)circle(0.07);\fill(u)circle(0.07);\fill(l)circle(0.07);\node at(r){$\times$};\node at(0.3,0){$o$};
\end{tikzpicture}
 \hspace{7mm}%
\begin{tikzpicture}[baseline=0mm]
\coordinate(0)at(0,0);\coordinate(u)at(90:1);\coordinate(r)at(-30:1);\coordinate(l)at(210:1);
\draw(0)--(u);\draw(0)--(l);\draw(0)--(r);
\draw(0)to[out=-60,in=-120,relative]node[pos=0.8]{\rotatebox{20}{\footnotesize $\bowtie$}}(u);
\fill(0)circle(0.07);\fill(u)circle(0.07);\node at(l){$\times$};\node at(r){$\times$};\node at(0.3,0){$o$};
\end{tikzpicture}
 \hspace{7mm}%
\begin{tikzpicture}[baseline=0mm]
\coordinate(0)at(0,0);\coordinate(u)at(90:1);\coordinate(r)at(-30:1);\coordinate(l)at(210:1);
\draw(0)--(u);\draw(0)--(l);\draw(0)--(r);
\fill(0)circle(0.07);\node at(u){$\times$};\node at(l){$\times$};\node at(r){$\times$};\node at(0.3,0.1){$o$};
\end{tikzpicture}
\caption{A complete list of exceptional triangles of orbifolds except for changing all tags at $o$, where each exceptional triangle forms a tagged triangulation of a sphere}
\label{Fig:exceptional-triangles}
\end{figure}
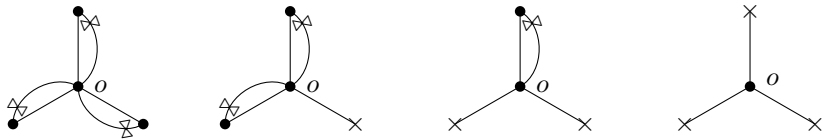
%%%

We can define \emph{flips} of tagged triangulations in the same way as ones of ideal triangulations. In particular, any tagged arc can be flipped.

\begin{theorem}[{\cite[Theorem 4.2]{FeST12a}\cite[Theorem 7.9 and Proposition 7.10]{FoST08}}]\label{thm:transitivity}
If $\cO$ is not an orbifold with empty boundary and exactly one puncture, then any two tagged triangulations of $\cO$ are connected by a sequence of flips. If $\cO$ is an orbifold with empty boundary and exactly one puncture, then two tagged triangulations of $\cO$ are connected by a sequence of flips if and only if all their tags at the puncture are the same.
\end{theorem}

%%%
%%%
%%% Lamination %%%
\subsection{Laminations on weighted orbifolds and their shear coordinates}\label{subsec:lam}

A \emph{weighted orbifold} $\cO^w$ is an orbifold $\cO=(S,M,Q)$ together with weights of orbifold points given by a map
\[
w : Q \rightarrow \left\{\frac{1}{2},2\right\}.
\]
We also say that a pending arc has \emph{weight $w(p)$} if it is incident to an orbifold point $p$.

A \emph{laminate} of $\cO^w$ is a non-self-intersecting curve in $S$ which is either
\begin{itemize}
\item a closed curve, or
\item a non-closed curve with each end being one of the following:
 \begin{itemize}
 \item an unmarked point on $\partial S$;
 \item an orbifold point with weight $\frac{1}{2}$;
 \item a spiral around a puncture (either clockwise or counterclockwise),
 \end{itemize}
\end{itemize}
and the following curves are not allowed (see Figure \ref{fignonlam}):
\begin{itemize}
\item a curve cutting out a disk with at most one puncture or orbifold point in total;
\item a curve with two endpoints on $\partial S$ such that it is isotopic to a segment of $\partial S$ containing at most one marked point;
\item a curve whose both ends are spirals around a common puncture in the same direction such that it does not enclose anything else;
\item a curve whose both endpoints are a common orbifold point.
\end{itemize}

%%%
%%%
\begin{figure}[htp]
\begin{minipage}{0.48\textwidth}
\centering
\begin{tikzpicture}[baseline=0mm]
\coordinate(l)at(-1,0.6);\coordinate(r)at(1,0.6);\coordinate(p)at(1,-0.8);\coordinate(q)at(1,-1.5);
\draw(2,2)--(-2,2)--(-2,-2)--(2,-2)--(2,2);\draw[pattern=north east lines](-0.3,-0.8)circle(0.4);
\draw[blue](2,-1.5)--(q);\node at(1.3,-0.8){$p$};\node at(0.7,-1.5){$q$};
\draw[blue](-2,1)arc(-90:0:1);\draw[blue](l)circle(0.4);\draw[blue](-1.4,-1.3)circle(0.4);
\draw[blue](1,1.1)arc(90:-90:0.4);\draw[blue](1,0.3)arc(-90:-280:0.24);
\draw[blue](0.6,0.6)..controls(0.6,1)and(1.2,1)..(1.2,0.6);\draw[blue](1.2,0.6)arc(0:-120:0.18);
\draw[blue](0.6,0.6)arc(-0:-270:0.5);\draw[blue](0.1,1.1)..controls(0.4,1.13)and(0.8,1.13)..(1,1.1);
\draw[blue](1,2)arc(0:-180:0.5);
\draw[blue](p)..controls(0.5,-0.4)and(0,-0.2)..(-0.3,-0.2);\draw[blue](p)..controls(0.5,-1.2)and(0,-1.4)..(-0.3,-1.4);
\draw[blue](-0.3,-0.2)arc(90:270:0.6);
\fill(l)circle(0.07);\fill(r)circle(0.07);\fill(2,2)circle(0.07);\fill(-2,2)circle(0.07);\fill(-2,-2)circle(0.07);\fill(2,-2)circle(0.07);\fill(-0.3,-0.4)circle(0.07);
\node at(p){$\times$};\node at(q){$\times$};
\end{tikzpicture}
\caption{Curves which are not laminates, where $w(p)=\frac{1}{2}$ and $w(q)=2$}
\label{fignonlam}
\end{minipage}
\begin{minipage}{0.48\textwidth}
\centering
\begin{tikzpicture}[baseline=0mm]
\coordinate(l)at(-1,0.6);\coordinate(r)at(1,0.6);\coordinate(p)at(1,-0.8);\coordinate(q)at(1,-1.5);
\draw(2,2)--(-2,2)--(-2,-2)--(2,-2)--(2,2);\draw[pattern=north east lines](-0.3,-0.8)circle(0.4);
\node at(1,-0.5){$p$};\node at(1.3,-1.5){$r$};
\draw[blue](-2,1.5)--(2,1.5);\draw[blue](-0.3,-0.8)circle(0.7);\draw[blue](-0.3,-0.8)circle(0.55);
\draw[blue](-1,0.8)..controls(-0.5,0.8)and(0.5,0.9)..(1,0.9);
\draw[blue](-1,0.8)arc(90:270:0.2);\draw[blue](-1,0.4)arc(-90:60:0.18);
\draw[blue](1,0.9)arc(90:0:0.3);\draw[blue](1.3,0.6)arc(0:-100:0.27);
\draw[blue](2,0)..controls(1,0)and(0.8,0.3)..(0.8,0.6);
\draw[blue](0.8,0.6)arc(180:-100:0.2);
\draw[blue](-1.3,-2)..controls(-1.3,0)and(-0.75,0.2)..(-0.3,0.2);
\draw[blue](0.7,-2)..controls(0.7,0)and(0.15,0.2)..(-0.3,0.2);
\draw[blue](p)--(q);
\fill(l)circle(0.07);\fill(r)circle(0.07);\fill(2,2)circle(0.07);\fill(-2,2)circle(0.07);\fill(-2,-2)circle(0.07);\fill(2,-2)circle(0.07);\fill(-0.3,-0.4)circle(0.07);
\node at(p){$\times$};\node at(q){$\times$};
\end{tikzpicture}
\caption{A lamination, where $w(p)=w(r)=\frac{1}{2}$}
\label{figlam}
\end{minipage}
\end{figure}
%%%

We say that two laminates of $\cO^w$ are \emph{compatible} if they do not intersect and they are not incident to a common orbifold point. Note that a laminate with at least one endpoint being an orbifold point is not compatible with itself. A finite collection of pairwise compatible laminates of $\cO^w$ is called a \emph{lamination} on $\cO^w$ (see Figure \ref{figlam}). To $\cO^w$, we associate a marked surface (see \cite[Definition 6.2]{FeST12a} and Figure \ref{fig:hat}).

\begin{definition}\label{def:widehat}
Let $\gamma$ be an ideal/tagged arc of $\cO^w$, $\ell$ a laminate of $\cO^w$, and $C$ a collection of ideal/tagged arcs or laminates of $\cO^w$.
\begin{enumerate}
\item The \emph{associated marked surface} $\hat{\cO}^w=(S,\hat{M},\emptyset)$ is obtained from $\cO^w$ by replacing each orbifold point $p$ with a puncture $\hat{p}$, that is, $M \subseteq \hat{M}$ and $|\hat{M}|=|M|+|Q|$. We also denote by $\hat{o}$ the marked point in $\hat{M}$ corresponding to $o\in M$.
\item If $\gamma$ is a pending ideal (resp., tagged) arc, then $\hat{\gamma}$ is a self-folded triangle (resp., a pair of conjugate arcs) $\{\gamma',\gamma''\}$ of $\hat{\cO}^w$ such that the underlying curves of $\gamma$ and $\gamma'$ coincide. Otherwise, $\hat{\gamma}$ is the corresponding ideal/tagged arc of $\hat{\cO}^w$.
\item A laminate $\hat{\ell}$ of $\hat{\cO}^w$ is obtained from $\ell$ by replacing its ends incident to an orbifold point $p$ with spirals around $\hat{p}$ clockwise if they exist,.
\item A collection $\hat{C}$ of ideal/tagged arcs or laminates of $\hat{\cO}^w$ is obtained from $C$ by replacing $c \in C$ with $\hat{c}$.
\end{enumerate}
\end{definition}

%%%
\begin{figure}[htp]
\begin{tikzpicture}[baseline=4mm]
\coordinate(d)at(0,0);\coordinate(p)at(0,1);
\draw(d)--node[right]{$\gamma$}(p)node[above]{$p$};
\fill(d)circle(0.07);\node at(p){$\times$};
\end{tikzpicture}
\hspace{2mm}$\xrightarrow{\widehat{(-)}}$\hspace{2mm}
\begin{tikzpicture}[baseline=4mm]
\coordinate(d)at(0,0);\coordinate(p)at(0,1);
\draw(d)--node[fill=white,inner sep=1]{$\gamma'$}(p);
\draw(d)to[out=30,in=0]node[right]{$\gamma''$}(0,1.6);\draw(d)to[out=150,in=180](0,1.6);
\fill(d)circle(0.07);\fill(p)circle(0.07)node[above]{$\hat{p}$};
\end{tikzpicture}
\hspace{1mm}or\hspace{1mm}
\begin{tikzpicture}[baseline=4mm]
\coordinate(d)at(0,0);\coordinate(p)at(0,1);
\draw(d)--node[left]{$\gamma'$}(p);
\draw(d)to[out=10,in=-10]node[right]{$\gamma''$}node[pos=0.8]{\rotatebox{40}{\footnotesize $\bowtie$}}(p);
\fill(d)circle(0.07);\fill(p)circle(0.07)node[above]{$\hat{p}$};
\end{tikzpicture}
\hspace{13mm}
\begin{tikzpicture}[baseline=4mm]
\coordinate(d)at(0,0.2);\coordinate(dd)at(0,-0.2);\coordinate(p)at(0,1);
\draw(d)--node[right]{$\ell$}(p)node[above]{$p$};
\node at(p){$\times$};\draw[dotted](d)--(dd);
\end{tikzpicture}
\hspace{2mm}$\xrightarrow{\widehat{(-)}}$\hspace{2mm}
\begin{tikzpicture}[baseline=4mm]
\coordinate(d)at(0,0.2);\coordinate(dd)at(0,-0.2);\coordinate(p)at(0,1);
\draw(d)..controls(0,0.5)and(-0.2,0.7)..(-0.2,1);\draw(-0.2,1)arc(180:-110:0.2);
\fill(p)circle(0.07);\draw[dotted](d)--(dd);
\node at(0.3,0.6){$\hat{\ell}$};\node at(0,1.4){$\hat{p}$};
\end{tikzpicture}
\caption{Applying the map $\widehat{(-)}$ for a pending arc $\gamma$ and a laminate $\ell$, where $\hat{\gamma}=\{\gamma',\gamma''\}$}
\label{fig:hat}
\end{figure}
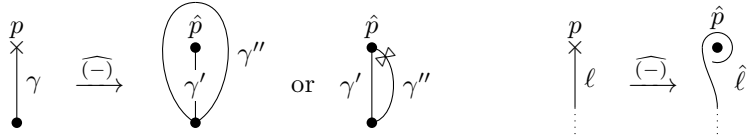

\begin{proposition}\label{prop:widehat well-def}
Let $T$ be a partial tagged triangulation and $L$ a lamination of $\cO^w$. Then $\hat{T}$ is a partial tagged triangulation of $\hat{\cO}^w$ and $\hat{L}$ is a lamination of $\hat{\cO}^w$. In addition, if $T$ is a tagged triangulation, then $\hat{T}$ is so.
\end{proposition}

\begin{proof}
The map $\widehat{(-)}$ keeps the compatibility of tagged arcs and laminates, and replaces triangles of $\cO^w$ with triangles of $\hat{\cO}^w$ (see Figures \ref{Fig:non-exceptional-triangles} and \ref{Fig:exceptional-triangles}). Thus the assertions hold.
\end{proof}

Let $\ell$ be a laminate of $\cO^w$. For an ideal/tagged triangulation $T$ of $\cO^w$, we define the \emph{shear coordinate $b_{\gamma,T}(\ell)$ of $\ell$} with respect to $\gamma \in T$ (see \cite[Section 6]{FeST12a} and \cite[Definitions 12.2 and 13.1]{FT18}).

First, we assume that $\cO^w$ has no orbifold points and $T$ is an ideal triangulation. If $\gamma \in T$ is not inside a self-folded triangle of $T$, then $b_{\gamma,T}(\ell)$ is defined by a sum of contributions from all intersections of $\gamma$ and $\ell$ as follows: Such an intersection contributes $+1$ (resp., $-1$) to $b_{\gamma,T}(\ell)$ if a segment of $\ell$ cuts through the quadrilateral surrounding $\gamma$ as in the left (resp., right) diagram of Figure \ref{fig:SZ}. Suppose that $\gamma \in T$ is inside a self-folded triangle $\{\gamma,\gamma'\}$ of $T$, where $\gamma'$ is a loop enclosing exactly one puncture $p$. Then we define $b_{\gamma,T}(\ell) = b_{\gamma',T}(\ell^{(p)})$, where $\ell^{(p)}$ is a laminate obtained from $\ell$ by changing the directions of its spirals at $p$ if they exist.

%%% fig:SZ
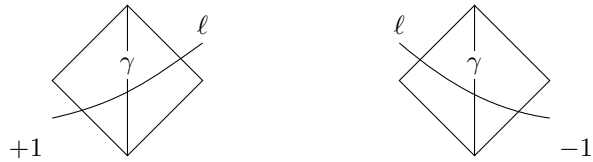
\begin{figure}[htp]
\centering
$+1$
\begin{tikzpicture}[baseline=0mm]
\coordinate(u)at(0,2);
\coordinate(l)at(-1,1);
\coordinate(r)at(1,1);
\coordinate(d)at(0,0);
\coordinate(s)at(-1,0.5);
\coordinate(t)at(1,1.5);
\draw(u)--(l)--(d)--(r)--(u)--node[fill=white,inner sep=2,pos=0.4]{$\gamma$}(d);
\draw(s)..controls (-0.1,0.7)and(0.3,1)..(t)node[above]{$\ell$};
\end{tikzpicture}
   \hspace{20mm}
\begin{tikzpicture}[baseline=0mm]
\coordinate(u)at(0,2);
\coordinate(l)at(-1,1);
\coordinate(r)at(1,1);
\coordinate(d)at(0,0);
\coordinate(s)at(-1,1.5);
\coordinate(t)at(1,0.5);
\draw(u)--(l)--(d)--(r)--(u)--node[fill=white,inner sep=2,pos=0.4]{$\gamma$}(d);
\draw(s)node[above]{$\ell$}..controls (-0.3,0.9)and(0.3,0.6)..(t);
\end{tikzpicture}
 $-1$
\caption{The contribution from a segment of the laminate $\ell$ on the left (resp., right) is $+1$ (resp., $-1$)}
\label{fig:SZ}
\end{figure}

%%%
Next, we assume that $\cO^w$ has no orbifold points and $T$ is a tagged triangulation. If there is an ideal triangulation $T^0$ such that $T=\iota(T^0)$, then we define $b_{\gamma,T}(\ell) = b_{\gamma^0,T^0}(\ell)$, where $\gamma=\iota(\gamma^0)$. For an arbitrary $T$, we can obtain a tagged triangulation $T^{(p_1\cdots p_m)}$ from $T$ by changing all tags at punctures $p_1,\ldots,p_m$ (possibly $m=0$), in such a way that there is a unique ideal triangulation $T^0$ such that $T^{(p_1\cdots p_m)}=\iota(T^0)$ (see \cite[Remark 3.11]{MSW11}). Then we define $b_{\gamma,T}(\ell) = b_{\gamma^{(p_1\cdots p_m)},T^{(p_1\cdots p_m)}}\bigl((\cdots((\ell^{(p_1)})^{(p_2)})\cdots)^{(p_m)}\bigr)$, where $\gamma^{(p_1\cdots p_m)}$ is a tagged arc of $T^{(p_1\cdots p_m)}$ corresponding to $\gamma$.

Finally, we consider any $\cO^w$ and $T$. Since the associated marked surface $\hat{\cO}^w$ has no orbifold points, we can define
\begin{equation}\label{eq:shear coordinates}
b_{\gamma,T}(\ell)=\left\{
\begin{array}{ll}
 b_{\hat{\gamma},\hat{T}}(\hat{\ell}) & \text{if $\gamma$ is not a pending arc},\\
 b_{\gamma',\hat{T}}(\hat{\ell})+b_{\gamma'',\hat{T}}(\hat{\ell}) & \text{if $\gamma$ is a pending arc with weight $\frac{1}{2}$ and $\hat{\gamma}=\{\gamma',\gamma''\}$},\\
 \frac{1}{2}\bigl(b_{\gamma',\hat{T}}(\hat{\ell})+b_{\gamma'',\hat{T}}(\hat{\ell})\bigr) & \text{if $\gamma$ is a pending arc with weight $2$ and $\hat{\gamma}=\{\gamma',\gamma''\}$}.
\end{array}\right.
\end{equation}

For a collection $L=L' \sqcup \{\ell\}$ of laminates of $\cO^w$, the \emph{shear coordinate $b_{\gamma,T}(L)$ of $L$} with respect to $\gamma \in T$ is inductively defined by
\[
 b_{\gamma,T}(L) = b_{\gamma,T}(L') + b_{\gamma,T}(\ell).
\]
We denote by $b_T(L)$ the vector $(b_{\gamma,T}(L))_{\gamma \in T} \in \bZ^{|T|}$, and by $C(L)$ the cone in $\bR^{|T|}$ spanned by $b_T(\ell)$ for all $\ell \in L$, that is, it is given by $C(L)=\{\sum_{\ell\in L}a_{\ell}b_T(\ell)\mid a_{\ell}\in\bR_{\ge0}\}$.

%%%
\begin{theorem}[{\cite[Theorem 6.7]{FeST12a}\cite[Theorems 12.3 and 13.6]{FT18}}]\label{thm:thurston}
Let $T$ be a tagged triangulation of $\cO^w$. The map sending laminations $L$ to $b_T(L)$ induces a bijection between the set of laminations on $\cO^w$ and $\bZ^{|T|}$.
\end{theorem}

%%% ex:digon
\begin{example}[{\cite[Example 2.4]{Y20}}]\label{ex:digon}
For a digon $\cO^w$ with exactly one puncture $p$ and no orbifold points, all laminates are given as follows:
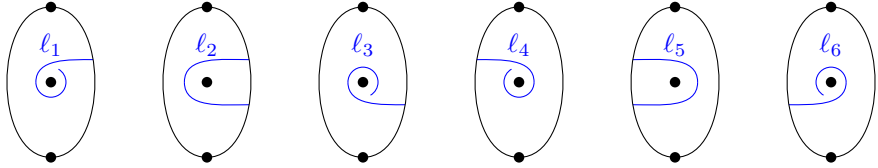
\[
\begin{tikzpicture}[baseline=0mm]
\coordinate(c)at(0,0);\coordinate(u)at(0,1);\coordinate(d)at(0,-1);
\draw[blue](0.56,0.3)to[out=180,in=90](-0.2,0);
\draw[blue](-0.2,0)arc(-180:60:0.2);
\draw(d)to[out=180,in=180](u);\draw(d)to[out=0,in=0](u);
\fill(c)circle(0.07);\fill(u)circle(0.07);\fill(d)circle(0.07);
\node[blue]at(0,0.5){$\ell_1$};
\end{tikzpicture}
\hspace{5mm}
\begin{tikzpicture}[baseline=0mm]
\coordinate(c)at(0,0);\coordinate(u)at(0,1);\coordinate(d)at(0,-1);
\draw[blue](0.56,0.3)to[out=180,in=90](-0.3,0);
\draw[blue](-0.3,0)to[out=-90,in=180](0.56,-0.3);
\draw(d)to[out=180,in=180](u);\draw(d)to[out=0,in=0](u);
\fill(c)circle(0.07);\fill(u)circle(0.07);\fill(d)circle(0.07);
\node[blue]at(0,0.5){$\ell_2$};
\end{tikzpicture}
\hspace{5mm}
\begin{tikzpicture}[baseline=0mm]
\coordinate(c)at(0,0);\coordinate(u)at(0,1);\coordinate(d)at(0,-1);
\draw[blue](0.56,-0.3)to[out=180,in=-90](-0.2,0);
\draw[blue](-0.2,0)arc(180:-60:0.2);
\draw(d)to[out=180,in=180](u);\draw(d)to[out=0,in=0](u);
\fill(c)circle(0.07);\fill(u)circle(0.07);\fill(d)circle(0.07);
\node[blue]at(0,0.5){$\ell_3$};
\end{tikzpicture}
\hspace{5mm}
\begin{tikzpicture}[baseline=0mm]
\coordinate(c)at(0,0);\coordinate(u)at(0,1);\coordinate(d)at(0,-1);
\draw[blue](-0.56,0.3)to[out=0,in=90](0.2,0);
\draw[blue](0.2,0)arc(0:-240:0.2);
\draw(d)to[out=180,in=180](u);\draw(d)to[out=0,in=0](u);
\fill(c)circle(0.07);\fill(u)circle(0.07);\fill(d)circle(0.07);
\node[blue]at(0,0.5){$\ell_4$};
\end{tikzpicture}
\hspace{5mm}
\begin{tikzpicture}[baseline=0mm]
\coordinate(c)at(0,0);\coordinate(u)at(0,1);\coordinate(d)at(0,-1);
\draw[blue](-0.56,0.3)to[out=0,in=90](0.3,0);
\draw[blue](0.3,0)to[out=-90,in=0](-0.56,-0.3);
\draw(d)to[out=180,in=180](u);\draw(d)to[out=0,in=0](u);
\fill(c)circle(0.07);\fill(u)circle(0.07);\fill(d)circle(0.07);
\node[blue]at(0,0.5){$\ell_5$};
\end{tikzpicture}
\hspace{5mm}
\begin{tikzpicture}[baseline=0mm]
\coordinate(c)at(0,0);\coordinate(u)at(0,1);\coordinate(d)at(0,-1);
\draw[blue](-0.56,-0.3)to[out=0,in=-90](0.2,0);
\draw[blue](0.2,0)arc(0:240:0.2);
\draw(d)to[out=180,in=180](u);\draw(d)to[out=0,in=0](u);
\fill(c)circle(0.07);\fill(u)circle(0.07);\fill(d)circle(0.07);
\node[blue]at(0,0.5){$\ell_6$};
\end{tikzpicture}
\]
We consider the tagged triangulation
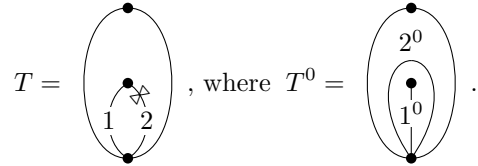
\[
 T=
\begin{tikzpicture}[baseline=-1mm]
\coordinate(c)at(0,0);\coordinate(u)at(0,1);\coordinate(d)at(0,-1);
\draw(d)to[out=180,in=180](u);\draw(d)to[out=0,in=0](u);
\draw(d)to[out=150,in=-150]node[fill=white,inner sep=2]{$1$}(c);
\draw(d)to[out=30,in=-30]node[pos=0.8]{\rotatebox{40}{\footnotesize $\bowtie$}}node[fill=white,inner sep=2]{$2$}(c);
\fill(c)circle(0.07);\fill(u)circle(0.07);\fill(d)circle(0.07);
\end{tikzpicture}
\text{, where}\hspace{2mm}
 T^0=
\begin{tikzpicture}[baseline=-1mm]
\coordinate(c)at(0,0);\coordinate(u)at(0,1);\coordinate(d)at(0,-1);
\draw(d)to[out=180,in=180](u);\draw(d)to[out=0,in=0](u);
\draw(d)--node[pos=0.6,fill=white,inner sep=1]{$1^0$}(c);
\draw(d)..controls(0.5,-0.3)and(0.3,0.3)..(0,0.3);\draw(d)..controls(-0.5,-0.3)and(-0.3,0.3)..(0,0.3)node[above]{$2^0$};
\fill(c)circle(0.07);\fill(u)circle(0.07);\fill(d)circle(0.07);
\end{tikzpicture}.
\]
The shear coordinate $b_{2,T}(\ell_1)$ is given by $b_{2^0,T^0}(\ell_1)=-1$. Since $\ell_3^{(p)}=\ell_1$, we have the equalities
\[
b_{1,T}(\ell_3) = b_{1^0,T^0}(\ell_3) = b_{2^0,T^0}(\ell_3^{(p)}) = b_{2^0,T^0}(\ell_1) = -1.
\]
Similarly, for $i \in \{1,2\}$ and $j \in \{1,\ldots,6\}$, the shear coordinates $b_{i,T}(\ell_j)$ and $b_{T}(\ell_j)$ are given as follows:
\[
\def\arraystretch{1.5}
\begin{tabular}{c||c|c|c|c|c|c}
 $i$\ \textbackslash\ $j$ & $1$ & $2$ & $3$ & $4$ & $5$ & $6$ \\\hline\hline
 $1$ & $0$ & $-1$ & $-1$ & $0$ & $1$ & $1$ \\\hline
 $2$ & $-1$ & $-1$ & $0$ & $1$ & $1$ & $0$
\end{tabular}
\hspace{10mm}
\begin{tikzpicture}[baseline=0mm,scale=1]
\coordinate (0) at (0,0);\coordinate (x) at (1,0);\coordinate (-x) at (-1,0);
\coordinate (y) at (0,1);\coordinate (-y) at (0,-1);
\draw[->] (0)--(x) node[right]{$b_T(\ell_6)$};
\draw (0)--(-x) node[left]{$b_T(\ell_3)$};
\draw[->] (0)--(y) node[above]{$b_T(\ell_4)$};
\draw (0)--(-y) node[below]{$b_T(\ell_1)$};
\draw (0)--(1,1) node[right]{$b_T(\ell_5)$};
\draw (0)--(-1,-1) node[left]{$b_T(\ell_2)$};
\end{tikzpicture}
\]
All laminations on $\cO^w$ are given by $\{m \ell_j,n \ell_{j+1}\}$ for $j \in \{1,\ldots,6\}$ and $m, n \in \bZ_{\ge 0}$, where $\ell_7 = \ell_1$. Since $C(\{\ell_{j-1},\ell_j\}) \cap C(\{\ell_j,\ell_{j+1}\})=C(\{\ell_j\})$ and $b_T$ induces a bijection
\[
b_T : \{\{m \ell_j,n \ell_{j+1}\} \mid m, n \in \bZ_{\ge 0}\} \leftrightarrow C(\{\ell_j,\ell_{j+1}\}) \cap  \bZ^2,
\]
there is a bijection between the set of laminations on $\cO^w$ and $\bZ^2$.
\end{example}
%%%

%%% ex:
\begin{example}\label{ex:monogon}
Let $\cO^w$ be a monogon with no punctures and exactly two orbifold points $p$ and $q$ such that $w(p)=2$ and $w(q)=\frac{1}{2}$. All laminates of $\cO^w$ are given as follows:
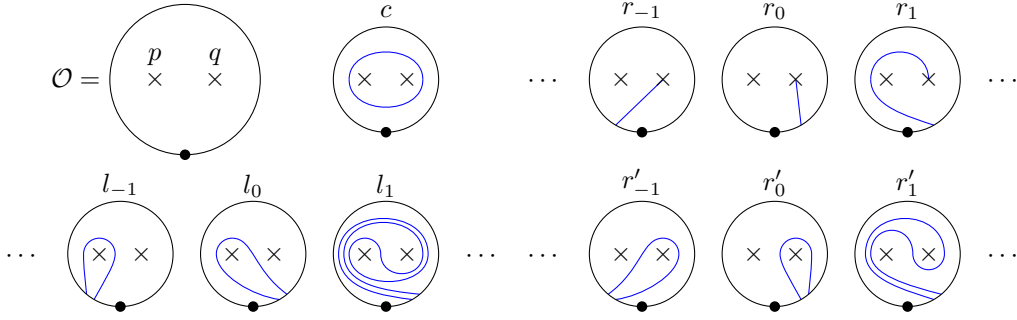
\[
\begin{tabular}{cccccccccc}
\multicolumn{3}{c}{
$\cO=$
\begin{tikzpicture}[baseline=-1mm,scale=1]
\coordinate(d)at(0,-1);\coordinate(l)at(-0.4,0);\coordinate(r)at(0.4,0);
\draw(0,0)circle(1);\fill(d)circle(0.07);\node at(l){$\times$};\node at(r){$\times$};\node at(-0.4,0.3){$p$};\node at(0.4,0.3){$q$};
\end{tikzpicture}
}&
\begin{tikzpicture}[baseline=-1mm,scale=0.7]
\coordinate(d)at(0,-1);\coordinate(l)at(-0.4,0);\coordinate(r)at(0.4,0);
\draw[blue](-0.7,0)..controls(-0.7,0.7)and(0.7,0.7)..(0.7,0);\draw[blue](-0.7,0)..controls(-0.7,-0.7)and(0.7,-0.7)..(0.7,0);
\draw(0,0)circle(1);\fill(d)circle(0.1);\node at(l){$\times$};\node at(r){$\times$};\node at(0,1.3){$c$};
\end{tikzpicture}
&&$\cdots$&
\begin{tikzpicture}[baseline=-1mm,scale=0.7]
\coordinate(d)at(0,-1);\coordinate(l)at(-0.4,0);\coordinate(r)at(0.4,0);\coordinate(ld)at(-120:1);\coordinate(rd)at(-60:1);
\draw[blue](r)--(ld);
\draw(0,0)circle(1);\fill(d)circle(0.1);\node at(l){$\times$};\node at(r){$\times$};\node at(0,1.3){$r_{-1}$};\node at(0,-1.3){};
\end{tikzpicture}
&
\begin{tikzpicture}[baseline=-1mm,scale=0.7]
\coordinate(d)at(0,-1);\coordinate(l)at(-0.4,0);\coordinate(r)at(0.4,0);\coordinate(ld)at(-120:1);\coordinate(rd)at(-60:1);
\draw[blue](r)--(rd);
\draw(0,0)circle(1);\fill(d)circle(0.1);\node at(l){$\times$};\node at(r){$\times$};\node at(0,1.3){$r_0$};
\end{tikzpicture}
&
\begin{tikzpicture}[baseline=-1mm,scale=0.7]
\coordinate(d)at(0,-1);\coordinate(l)at(-0.4,0);\coordinate(r)at(0.4,0);\coordinate(ld)at(-120:1);\coordinate(rd)at(-60:1);
\draw[blue](r)..controls(0.4,0.7)and(-0.7,0.7)..(-0.7,0);\draw[blue](-0.7,0)..controls(-0.7,-0.5)and(0,-0.7)..(rd);
\draw(0,0)circle(1);\fill(d)circle(0.1);\node at(l){$\times$};\node at(r){$\times$};\node at(0,1.3){$r_1$};
\end{tikzpicture}
&$\cdots$
\\
$\cdots$&
\begin{tikzpicture}[baseline=-1mm,scale=0.7]
\coordinate(d)at(0,-1);\coordinate(l)at(-0.4,0);\coordinate(r)at(0.4,0);\coordinate(lld)at(-130:1);\coordinate(ld)at(-120:1);
\draw[blue](ld)..controls(-0.3,-0.5)and(-0.1,-0.2)..(-0.1,0);\draw[blue](lld)to[out=95,in=-90](-0.7,0);
\draw[blue](-0.1,0)arc(0:180:0.3);
\draw(0,0)circle(1);\fill(d)circle(0.1);\node at(l){$\times$};\node at(r){$\times$};\node at(0,1.3){$l_{-1}$};
\end{tikzpicture}
&
\begin{tikzpicture}[baseline=-1mm,scale=0.7]
\coordinate(d)at(0,-1);\coordinate(l)at(-0.4,0);\coordinate(r)at(0.4,0);\coordinate(rrd)at(-50:1);\coordinate(rd)at(-60:1);
\draw[blue](rd)..controls(0,-0.8)and(-0.7,-0.4)..(-0.7,0);\draw[blue](rrd)..controls(0,-0.3)and(0,0.3)..(-0.4,0.3);
\draw[blue](-0.7,0)arc(180:90:0.3);
\draw(0,0)circle(1);\fill(d)circle(0.1);\node at(l){$\times$};\node at(r){$\times$};\node at(0,1.3){$l_0$};
\end{tikzpicture}
&
\begin{tikzpicture}[baseline=-1mm,scale=0.7]
\coordinate(d)at(0,-1);\coordinate(l)at(-0.4,0);\coordinate(r)at(0.4,0);\coordinate(rrd)at(-50:1);\coordinate(rd)at(-60:1);
\draw[blue](-0.7,0)..controls(-0.7,-0.7)and(0.8,-0.7)..(0.8,0);\draw[blue](0.8,0)..controls(0.8,0.9)and(-0.9,0.9)..(-0.9,0);
\draw[blue](-0.9,0)..controls(-0.9,-0.7)and(0,-0.8)..(rd);
\draw[blue](-0.1,0)..controls(-0.1,-0.5)and(0.7,-0.5)..(0.7,0);\draw[blue](0.7,0)..controls(0.7,0.8)and(-0.8,0.8)..(-0.8,0);
\draw[blue](-0.8,0)..controls(-0.8,-0.6)and(0,-0.7)..(rrd);
\draw[blue](-0.1,0)arc(0:180:0.3);
\draw(0,0)circle(1);\fill(d)circle(0.1);\node at(l){$\times$};\node at(r){$\times$};\node at(0,1.3){$l_1$};
\end{tikzpicture}
&$\cdots$&$\cdots$&
\begin{tikzpicture}[baseline=-1mm,scale=0.7]
\coordinate(d)at(0,-1);\coordinate(l)at(-0.4,0);\coordinate(r)at(0.4,0);\coordinate(ld)at(-120:1);\coordinate(lld)at(-130:1);
\draw[blue](ld)..controls(0,-0.8)and(0.7,-0.4)..(0.7,0);\draw[blue](lld)..controls(0,-0.3)and(0,0.3)..(0.4,0.3);
\draw[blue](0.7,0)arc(0:90:0.3);
\draw(0,0)circle(1);\fill(d)circle(0.1);\node at(l){$\times$};\node at(r){$\times$};\node at(0,1.35){$r_{-1}'$};
\end{tikzpicture}
&
\begin{tikzpicture}[baseline=-1mm,scale=0.7]
\coordinate(d)at(0,-1);\coordinate(l)at(-0.4,0);\coordinate(r)at(0.4,0);\coordinate(rrd)at(-50:1);\coordinate(rd)at(-60:1);
\draw[blue](rd)..controls(0.3,-0.5)and(0.1,-0.2)..(0.1,0);\draw[blue](rrd)to[out=85,in=-90](0.7,0);
\draw[blue](0.1,0)arc(180:0:0.3);
\draw(0,0)circle(1);\fill(d)circle(0.1);\node at(l){$\times$};\node at(r){$\times$};\node at(0,1.35){$r_0'$};
\end{tikzpicture}
&
\begin{tikzpicture}[baseline=-1mm,scale=0.7]
\coordinate(d)at(0,-1);\coordinate(l)at(-0.4,0);\coordinate(r)at(0.4,0);\coordinate(rrd)at(-50:1);\coordinate(rd)at(-60:1);
\draw[blue](0.7,0)..controls(0.7,0.9)and(-0.8,0.9)..(-0.8,0);\draw[blue](-0.8,0)..controls(-0.8,-0.5)and(0,-0.7)..(rd);
\draw[blue](0.1,0)..controls(0.1,0.6)and(-0.7,0.6)..(-0.7,0);\draw[blue](-0.7,0)..controls(-0.7,-0.4)and(0,-0.6)..(rrd);
\draw[blue](0.1,0)arc(-180:0:0.3);
\draw(0,0)circle(1);\fill(d)circle(0.1);\node at(l){$\times$};\node at(r){$\times$};\node at(0,1.35){$r_1'$};
\end{tikzpicture}
&
$\cdots$
\end{tabular}
\]
where $l_i=\sT_c^i(l_0)$, $r_i=\sT_c^i(r_0)$, and $r_i'=\sT_c^i(r_0')$ for $i\in\bZ$ and the Dehn twist $\sT_c$ along $c$ (see Subsection \ref{subsec:Dehn}). We consider the tagged triangulation
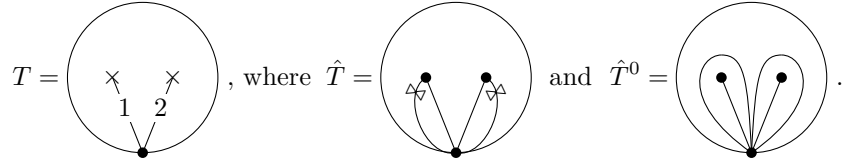
\[
 T=
\begin{tikzpicture}[baseline=-1mm]
\coordinate(d)at(0,-1);\coordinate(l)at(-0.4,0);\coordinate(r)at(0.4,0);
\draw(l)--node[fill=white,inner sep=2,pos=0.4]{$1$}(d)--node[fill=white,inner sep=2,pos=0.6]{$2$}(r);
\draw(0,0)circle(1);\fill(d)circle(0.07);\node at(l){$\times$};\node at(r){$\times$};
\end{tikzpicture}
\ \text{, where}\hspace{2mm}
\hat{T}=
\begin{tikzpicture}[baseline=-1mm]
\coordinate(d)at(0,-1);\coordinate(l)at(-0.4,0);\coordinate(r)at(0.4,0);
\draw(0,0)circle(1);\draw(l)--(d)--(r);
\draw(l)to[out=-80,in=-100,relative]node[pos=0.2]{\rotatebox{-25}{\footnotesize $\bowtie$}}(d);
\draw(r)to[out=80,in=100,relative]node[pos=0.2]{\rotatebox{25}{\footnotesize $\bowtie$}}(d);
\fill(d)circle(0.07);\fill(l)circle(0.07);\fill(r)circle(0.07);
\end{tikzpicture}
\ \text{ and }\ 
\hat{T}^0=
\begin{tikzpicture}[baseline=-1mm]
\coordinate(d)at(0,-1);\coordinate(l)at(-0.4,0);\coordinate(r)at(0.4,0);
\draw(0,0)circle(1);\draw(l)--(d)--(r);
\draw(d)..controls(-1,-0.3)and(-0.7,0.3)..(-0.4,0.3);\draw(d)..controls(-0.05,0)and(-0.1,0.3)..(-0.4,0.3);
\draw(d)..controls(1,-0.3)and(0.7,0.3)..(0.4,0.3);\draw(d)..controls(0.05,0)and(0.1,0.3)..(0.4,0.3);
\fill(d)circle(0.07);\fill(l)circle(0.07);\fill(r)circle(0.07);
\end{tikzpicture}\ .
\]
\begin{minipage}{0.7\textwidth}
Then we can get $b_T(c)=(-1,2)$, $b_T(r_i')=2 b_T(r_i)$, and
{\setlength\arraycolsep{0.5mm}
\begin{eqnarray*}
b_T(r_i)&=&\left\{
\begin{array}{ll}
b_T(r_{-1})-(i+1)b_T(c) & (i\le-2),\\
(0,1) & (i=-1),\\
(0,-1) & (i=0),\\
b_T(r_0)+ib_T(c) & (i\ge1),
\end{array} \right.
\\
b_T(l_i)&=&\left\{
\begin{array}{ll}
b_T(l_{-1})-2(i+1)b_T(c) & (i\le-2),\\
(1,0) & (i=-1),\\
(-1,0) & (i=0),\\
b_T(l_0)+2ib_T(c) & (i\ge1).
\end{array} \right.
\end{eqnarray*}}\hspace{-1.5mm}
Here, we denote by $\ell$ the vector $b_{T}(\ell)$ in the right diagram.
\end{minipage}
\begin{minipage}{0.28\textwidth}
\begin{tikzpicture}[baseline=0mm,scale=0.8]
\draw[dotted](1,-1)grid(-3,4);\node at(0,-1.2){};
\draw(-3,0)--(1,0);\draw(0,-1)--(0,4);\fill(-1,2)circle(0.1);
\fill(1,0)circle(0.1);\fill(-1,0)circle(0.1);\fill(0,-1)circle(0.1);\fill(0,1)circle(0.1);
\fill(-1,3)circle(0.1);\fill(-1,1)circle(0.1);\fill(-2,3)circle(0.1);%r
\fill(-1,4)circle(0.1);\fill(-3,4)circle(0.1);%l
\draw(0,0)--(-1,1);\draw(0,0)--(-2,3);\draw(0,0)--(-1,3);
\draw(0,0)--(-1,4);\draw(0,0)--(-3,4);\draw(0,0)--(-1,2);
\node at(-1.35,3.25){$r_{-2}$};\node at(-2.2,3.3){$r_2$};\node at(-1.3,4.3){$l_{-2}$};\node at(-2.8,4.3){$l_1$};\node at(-1.2,2.2){$c$};
\node at(-1.3,1.2){$r_1$};\node at(-1.2,-0.3){$l_0$};\node at(0.3,-0.8){$r_0$};\node at(0.7,-0.3){$l_{-1}$};\node at(0.45,1.2){$r_{-1}$};
\end{tikzpicture}
\end{minipage}\\
All laminations on $\cO^w$ are given by $\{c\}$, $\{k r_i, m r_i', n l_i\}$, and $\{k r_i, m r_i', n l_{i-1}\}$ for $i \in \bZ$, $k \in \{0, 1\}$, and $m, n \in \bZ_{\ge 0}$. Notice that the collections $\{r_i, r_i\}$ are not laminations. It is easy to check that $b_T$ induces a bijection between the set of laminations on $\cO^w$ and $\bZ^2$.
\end{example}

%%%
%%%
%%% subsec:kinds of lam
\subsection{Elementary, exceptional, semi-closed, and closed laminates}\label{subsec:kinds of lam}

For a tagged arc $\delta$ of $\cO^w$, we define an \emph{elementary laminate} $\se(\delta)$. First, if $\delta$ is a pending arc with weight $2$ connecting $o\in M$ and $q\in Q$, then we replace it with a loop at $o$, cutting out a monogon with no punctures and exactly one orbifold point $q$, whose both ends are tagged in the same way as $\delta$ at $o$. The loop is not a tagged arc, but we denote it by $\delta$ again. The laminate $\se(\delta)$ is given as follows:
\begin{itemize}
\item $\se(\delta)$ is a laminate running along $\delta$ in its small neighborhood;
\item If an endpoint of $\delta$ is a marked point $o$ on a component $C$ of $\partial S$, then the corresponding endpoint of $\se(\delta)$ is located near $o$ on $C$ in the clockwise direction as in the left diagram of Figure \ref{fig:elementary lam};
\item If an endpoint of $\delta$ is a puncture $p$, then the corresponding end of $\se(\delta)$ is a spiral around $p$ clockwise (resp., counterclockwise) if $\delta$ is tagged plain (resp., notched) at $p$ as in the center (resp., right) diagram of Figure \ref{fig:elementary lam};
\item If an endpoint of $\delta$ is an orbifold point with weight $\frac{1}{2}$, then the corresponding end of $\se(\delta)$ is the same as one of $\delta$ as in the center of Figure \ref{fig:elementary lam}.
\end{itemize}
See the right diagram of Figure \ref{fig:elementary lam} for the case that $\delta$ is a pending arc with weight $2$. It follows from the construction that the map $\se$ from the set of tagged arcs to the set of laminates is injective.

%%%
\begin{figure}[htp]
\centering
\begin{tikzpicture}[baseline=0mm]
\coordinate(u)at(0,0.5);\coordinate(l)at(-0.5,0);\coordinate(r)at(0.5,0);\coordinate(d)at(0,-0.5);
\coordinate(u1)at(3,0.5);\coordinate(l1)at(2.5,0);\coordinate(r1)at(3.5,0);\coordinate(d1)at(3,-0.5);
\draw(r)--node[above,pos=0.2]{$\delta$}(l1);
\draw(0.4,-0.3)..controls(1.5,-0.3)and(1.5,0.3)..node[above,pos=0.8]{$\se(\delta)$}(2.6,0.3);
\draw[pattern=north east lines](u)arc(90:-90:5mm);
\draw[pattern=north east lines](u1)arc(90:270:5mm);
\draw[dotted](u)arc(90:270:5mm);
\draw[dotted](u1)arc(90:-90:5mm);
\fill(r)circle(0.07);\fill(l1)circle(0.07);
\end{tikzpicture}
\hspace{10mm}
\begin{tikzpicture}[baseline=0mm]
\coordinate(l)at(0,0);\coordinate(r)at(2.2,0);
\draw(l)--node[above]{$\delta$}(r);
\draw(0,-0.3)..controls(0.5,-0.3)and(1.5,0)..node[below]{$\se(\delta)$}(r);
\draw(0,-0.3)arc(-90:-270:0.26);\draw(0,0.22)arc(90:-90:0.2);
\fill(l)circle(0.07);\node at(r){$\times$};\node at(2.5,0){$r$};
\end{tikzpicture}
\hspace{10mm}
\begin{tikzpicture}[baseline=0mm]
\coordinate(l)at(0,0);\coordinate(r)at(2.2,0);
\draw(l)--node[below,pos=0.3]{$\delta$}node[pos=0.15]{\rotatebox{90}{\footnotesize $\bowtie$}}(r);
\draw(0,0.3)--node[above,pos=0.7]{$\se(\delta)$}(2.5,0.3);\draw(2.5,-0.3)arc(-90:90:0.3);
\draw(0,0.25)..controls(1,0.25)and(1.5,-0.3)..(2.5,-0.3);
\draw(0,0.3)arc(90:270:0.26);\draw(0,-0.22)arc(-90:90:0.2);
\draw(0,0.25)arc(90:270:0.2);\draw(0,-0.15)arc(-90:90:0.14);
\fill(l)circle(0.07);\node at(r){$\times$};\node at(2.5,0){$q$};
\end{tikzpicture}
\caption{Elementary laminates of tagged arcs, where $w(r)=\frac{1}{2}$ and $w(q)=2$}
\label{fig:elementary lam}
\end{figure}
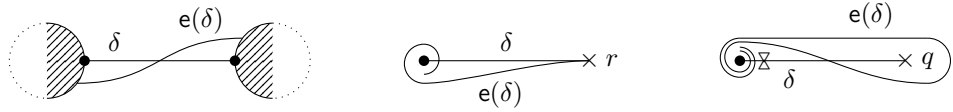
%%%

The following properties of elementary laminates are proved in \cite[Proposition 2.5]{Y20} for the case that $\cO^w$ has no orbifold points. Similarly, we can get the same statement for the general case.

%%%
\begin{proposition}\label{prop:properties of elementary}
(1) Let $\delta$ and $\delta'$ be tagged arcs such that $\delta^{\circ} \neq \delta'^{\circ}$. Then $\delta$ and $\delta'$ are compatible if and only if $\se(\delta)$ and $\se(\delta')$ are compatible.\par
(2) The map $\se$ induces a bijection between the set of partial tagged triangulations of $\cO^w$ without pairs of conjugate arcs and the set of laminations on $\cO^w$ consisting only of distinct elementary laminates.
\end{proposition}

\begin{proof}
(1) By the definition of $\se$, it is enough to consider neighborhoods of their endpoints. In particular, if $\delta$ and $\delta'$ have no common endpoints, the assertion holds. If they are incident to a common orbifold point, then they are not compatible and $\se(\delta)$ and $\se(\delta')$ are not compatible. Suppose that $\delta$ and $\delta'$ have at least one common endpoint which is not an orbifold point. Since $\delta^{\circ} \neq \delta'^{\circ}$, $\{\delta,\delta'\}$ is not a pair of conjugate arcs. Thus $\delta$ and $\delta'$ are compatible if and only if the tags of $\delta$ and $\delta'$ at each common endpoint are the same. By the definition of $\se$, it is equivalent that $\se(\delta)$ and $\se(\delta')$ are compatible.

(2) If two distinct tagged arcs $\delta$ and $\delta'$ satisfying $\delta^{\circ} = \delta'^{\circ}$ are compatible, then $\{\delta,\delta'\}$ is a pair of conjugate arcs, in which case $\se(\delta)$ and $\se(\delta')$ are not compatible. Thus the assertion follows from (1).
\end{proof}

\begin{definition}
A laminate whose both endpoints are orbifold points with weight $\frac{1}{2}$ is called \emph{semi-closed}. A laminate which is neither elementary, semi-closed, nor closed is called \emph{exceptional}.
\end{definition}

Exceptional laminates are characterized as follows (see Figure \ref{fig:exc}).

\begin{proposition}\label{prop:excep}
A laminate is exceptional if and only if it satisfies the following conditions:
\begin{itemize}
\item It encloses exactly one puncture and no orbifold points, or no punctures and exactly one orbifold point with weight $\frac{1}{2}$;
\item Its both ends are incident to a common boundary segment, or spirals around a common puncture in the same direction.
\end{itemize}
\end{proposition}

\begin{proof}
Let $\ell$ be a laminate which is neither semi-closed nor closed. If $\ell$ does not satisfy the conditions, then we can obtain a tagged arc from $\ell$ by applying the same transformation as $\se^{-1}$, that is, $\ell$ is elementary. Thus the assertion follows from the definitions of tagged arcs and the map $\se$.
\end{proof}

%%%
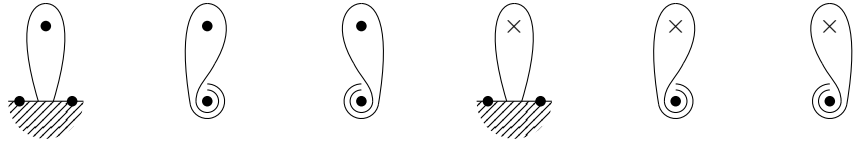
\begin{figure}[htp]
\centering
\begin{tikzpicture}[baseline=0mm]
\coordinate(0)at(0,0);\coordinate(p)at(0,1);
\draw(-0.1,0)..controls(-0.4,1)and(-0.2,1.3)..(0,1.3);
\draw(0.1,0)..controls(0.4,1)and(0.2,1.3)..(0,1.3); 
\draw(-0.5,0)--(0.5,0);
\fill(-0.35,0)circle(0.07);\fill(0.35,0)circle(0.07);\fill(p)circle(0.07);
\fill[pattern=north east lines](-0.5,0)arc(-180:0:0.5);
\end{tikzpicture}
 \hspace{10mm}
\begin{tikzpicture}[baseline=0mm]
\coordinate(0)at(0,0);\coordinate(p)at(0,1);
\draw(-170:0.23)..controls(-0.4,1)and(-0.2,1.3)..(0,1.3);
\draw(0.1,0.5)..controls(0.4,1)and(0.2,1.3)..(0,1.3);
\draw(0.1,0.5)..controls(0,0.3)and(-0.15,0.2)..(-0.15,0);
\draw(-170:0.23)arc(-170:90:0.23);
\draw(-0.15,0)arc(-180:90:0.15);
\fill(0,0)circle(0.07);\fill(p)circle(0.07);
\end{tikzpicture}
 \hspace{10mm}
\begin{tikzpicture}[baseline=0mm]
\coordinate(0)at(0,0);\coordinate(p)at(0,1);
\draw(-10:0.23)..controls(0.4,1)and(0.2,1.3)..(0,1.3);
\draw(-0.1,0.5)..controls(-0.4,1)and(-0.2,1.3)..(0,1.3);
\draw(-0.1,0.5)..controls(0,0.3)and(0.15,0.2)..(0.15,0);
\draw(-10:0.23)arc(-10:-270:0.23);
\draw(0.15,0)arc(0:-270:0.15);
\fill(0,0)circle(0.07);\fill(p)circle(0.07);
\end{tikzpicture}
\hspace{10mm}%with orbifold point
\begin{tikzpicture}[baseline=0mm]
\coordinate(0)at(0,0);\coordinate(p)at(0,1);
\draw(-0.1,0)..controls(-0.4,1)and(-0.2,1.3)..(0,1.3);
\draw(0.1,0)..controls(0.4,1)and(0.2,1.3)..(0,1.3); 
\draw(-0.5,0)--(0.5,0);
\fill(-0.35,0)circle(0.07);\fill(0.35,0)circle(0.07);\node at(p){$\times$};
\fill[pattern=north east lines](-0.5,0)arc(-180:0:0.5);
\end{tikzpicture}
 \hspace{10mm}
\begin{tikzpicture}[baseline=0mm]
\coordinate(0)at(0,0);\coordinate(p)at(0,1);
\draw(-170:0.23)..controls(-0.4,1)and(-0.2,1.3)..(0,1.3);
\draw(0.1,0.5)..controls(0.4,1)and(0.2,1.3)..(0,1.3);
\draw(0.1,0.5)..controls(0,0.3)and(-0.15,0.2)..(-0.15,0);
\draw(-170:0.23)arc(-170:90:0.23);
\draw(-0.15,0)arc(-180:90:0.15);
\fill(0,0)circle(0.07);\node at(p){$\times$};
\end{tikzpicture}
 \hspace{10mm}
\begin{tikzpicture}[baseline=0mm]
\coordinate(0)at(0,0);\coordinate(p)at(0,1);
\draw(-10:0.23)..controls(0.4,1)and(0.2,1.3)..(0,1.3);
\draw(-0.1,0.5)..controls(-0.4,1)and(-0.2,1.3)..(0,1.3);
\draw(-0.1,0.5)..controls(0,0.3)and(0.15,0.2)..(0.15,0);
\draw(-10:0.23)arc(-10:-270:0.23);
\draw(0.15,0)arc(0:-270:0.15);
\fill(0,0)circle(0.07);\node at(p){$\times$};
\end{tikzpicture}
\caption{Exceptional laminates, where the orbifold points have weight $\frac{1}{2}$}\label{fig:exc}
\end{figure}
%%%

To interpret shear coordinates of exceptional laminates as ones of elementary laminates, we introduce the following notations. For an exceptional laminate $\ell$ of $\cO^w$, elementary laminates $\ell_{\sf p}$ and $\ell_{\sf q}$ are given by
%%%
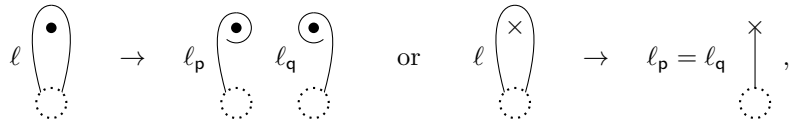
\[
 \ell
\begin{tikzpicture}[baseline=5mm]
\coordinate(0)at(0,0);\coordinate(p)at(0,1);
\draw(-0.1,0)..controls(-0.4,1)and(-0.2,1.3)..(0,1.3);
\draw(0.1,0)..controls(0.4,1)and(0.2,1.3)..(0,1.3);
\fill(p)circle(0.07);\filldraw[dotted,thick,fill=white](0)circle(0.2);
\end{tikzpicture}
 \hspace{4mm}\rightarrow\hspace{4mm}
 \ell_{\sf p}
\begin{tikzpicture}[baseline=5mm]
\coordinate(0)at(0,0);\coordinate(p)at(0,1);
\draw(-0.1,0)..controls(-0.4,1)and(-0.2,1.2)..(0,1.2);
\draw(0,1.2)arc(90:-130:0.2);
\fill(p)circle(0.07);\filldraw[dotted,thick,fill=white](0)circle(0.2);
\end{tikzpicture}
 \hspace{3mm}
 \ell_{\sf q}
\begin{tikzpicture}[baseline=5mm]
\coordinate(0)at(0,0);\coordinate(p)at(0,1);
\draw(0.1,0)..controls(0.4,1)and(0.2,1.2)..(0,1.2);
\draw(0,1.2)arc(90:310:0.2);
\fill(p)circle(0.07);\filldraw[dotted,thick,fill=white](0)circle(0.2);
\end{tikzpicture}
\hspace{7mm}\text{or}\hspace{7mm}
 \ell
\begin{tikzpicture}[baseline=5mm]
\coordinate(0)at(0,0);\coordinate(p)at(0,1);
\draw(-0.1,0)..controls(-0.4,1)and(-0.2,1.3)..(0,1.3);
\draw(0.1,0)..controls(0.4,1)and(0.2,1.3)..(0,1.3);
\node at(p){$\times$};\filldraw[dotted,thick,fill=white](0)circle(0.2);
\end{tikzpicture}
 \hspace{4mm}\rightarrow\hspace{4mm}
 \ell_{\sf p}=\ell_{\sf q}\ 
\begin{tikzpicture}[baseline=5mm]
\coordinate(0)at(0,0);\coordinate(p)at(0,1);
\draw(0)--(p);
\node at(p){$\times$};\filldraw[dotted,thick,fill=white](0)circle(0.2);
\end{tikzpicture}\ ,
\]
\begin{equation}\label{eq:ends}
 \text{where}\ 
\begin{tikzpicture}[baseline=-1mm]
\draw[dotted,thick] (0,0)circle(0.5);
\end{tikzpicture}
 =
\begin{tikzpicture}[baseline=-1mm]
\draw[dotted,thick] (0,0)circle(0.5);
\draw(-0.15,0)--(-0.25,0.4) (0.15,0)--(0.25,0.4) (-0.5,0)--(0.5,0);
\fill(-0.35,0)circle(0.07);  \fill(0.35,0)circle(0.07);
\fill[pattern=north east lines](-0.5,0)arc(-180:0:0.5);
\end{tikzpicture}
\ \text{or}\ 
\begin{tikzpicture}[baseline=-1mm]
\draw[dotted,thick] (0,0)circle(0.5);
\draw(-0.25,0.4)--(-0.25,0);
\draw(0.25,0.4)..controls(0.2,0.4)and(-0.15,0.5)..(-0.15,0);
\draw(-180:0.25)arc(-180:90:0.25);
\draw(-0.15,0)arc(-180:90:0.15);
\fill(0,0)circle(0.07);
\end{tikzpicture}
\ \text{or}\ 
\begin{tikzpicture}[baseline=-1mm]
\draw[dotted,thick] (0,0)circle(0.5);
\draw(0.25,0.4)--(0.25,0);
\draw(-0.25,0.4)..controls(-0.2,0.4)and(0.15,0.5)..(0.15,0);
\draw(0:0.25)arc(0:-270:0.25);
\draw(0.15,0)arc(0:-270:0.15);
\fill(0,0)circle(0.07);
\end{tikzpicture}\ .
\end{equation}
%%%
For a lamination $L$ on $\cO^w$, we denote by $L_{\sf pq}$ the collection of elementary laminates obtained from $L$ by replacing exceptional laminates $\ell \in L$ with $\ell_{\sf p}$ and $\ell_{\sf q}$.

%%% ex:digon
\begin{example}\label{ex:digon pq}
We keep the notations of Example \ref{ex:digon}. Then we have $(\ell_2)_{\sf p}=\ell_3$, $(\ell_2)_{\sf q}=\ell_1$, $(\ell_5)_{\sf p}=\ell_4$, and $(\ell_5)_{\sf q}=\ell_6$. In particular, for $j\in\{2,5\}$, we have
\begin{equation}\label{eq:digon pq}
b_T(\ell_j)=b_T((\ell_j)_{\sf p})+b_T((\ell_j)_{\sf q}).
\end{equation}
Moreover, we also have
\begin{equation}\label{eq:p=q}
b_{1,T}((\ell_j)_{\sf p})+b_{2,T}((\ell_j)_{\sf p})=b_{1,T}((\ell_j)_{\sf q})+b_{2,T}((\ell_j)_{\sf q}).
\end{equation}
\end{example}
%%%

In general, the same property as \eqref{eq:digon pq} holds for arbitrary exceptional laminates.

%%%
\begin{lemma}\label{lem:=p+q}
Let $T$ be a tagged triangulation of $\cO^w$. For an exceptional laminate $\ell$ of $\cO^w$, we have
\[
b_T(\ell)=b_T(\ell_{\sf p})+b_T(\ell_{\sf q}).
\]
\end{lemma}

\begin{proof}
We only need to prove
\begin{equation}\label{eq:=p+q}
 b_{\delta,T}(\ell)=b_{\delta,T}(\ell_{\sf p})+b_{\delta,T}(\ell_{\sf q})
\end{equation}
for any $\delta\in T$. For the case that $\cO^w$ has no orbifold points, it is given in \cite[Lemma 2.8]{Y20}. In general, since $\hat{\cO}^w$ has no orbifold points, \cite[Lemma 2.8]{Y20} gives an equality
\begin{equation}\label{eq:hat}
b_{\delta',\hat{T}}(\hat{\ell})=b_{\delta',\hat{T}}((\hat{\ell})_{\sf p})+b_{\delta',\hat{T}}((\hat{\ell})_{\sf q})
\end{equation}
for any $\delta'\in\hat{T}$. By Proposition \ref{prop:excep}, there is a unique puncture or orbifold point with weight $\frac{1}{2}$ enclosed by $\ell$, where we denote it by $p$. If $p$ is a puncture, then we have $(\hat{\ell})_{\sf p}=\widehat{(\ell_{\sf p})}$ and $(\hat{\ell})_{\sf q}=\widehat{(\ell_{\sf q})}$. Thus \eqref{eq:=p+q} follows from \eqref{eq:shear coordinates} and \eqref{eq:hat}. If $p$ is an orbifold point, then there is a pending arc $\varepsilon\in T$ incident to $p$ such that $\hat{\varepsilon}$ is a pair of conjugate arcs $\{\varepsilon',\varepsilon''\}$ of $\hat{T}$. Thus \eqref{eq:p=q} means that
\begin{equation}\label{eq:p+q 2}
b_{\varepsilon',\hat{T}}((\hat{\ell})_{\sf p})+b_{\varepsilon'',\hat{T}}((\hat{\ell})_{\sf p})=b_{\varepsilon',\hat{T}}((\hat{\ell})_{\sf q})+b_{\varepsilon'',\hat{T}}((\hat{\ell})_{\sf q})
\end{equation}
since there is a digon of $\hat{T}^0$ with exactly one puncture $\hat{p}$. 

Therefore, we can get \eqref{eq:=p+q} for $\delta=\varepsilon$ as follows:
\[\arraycolsep=0.5mm\def\arraystretch{1.7}\begin{array}{llll}
b_{\varepsilon,T}(\ell)&=&b_{\varepsilon',\hat{T}}(\hat{\ell})+b_{\varepsilon'',\hat{T}}(\hat{\ell})&(\text{by }\eqref{eq:shear coordinates})\\
&=&b_{\varepsilon',\hat{T}}((\hat{\ell})_{\sf p})+b_{\varepsilon',\hat{T}}((\hat{\ell})_{\sf q})+b_{\varepsilon'',\hat{T}}((\hat{\ell})_{\sf p})+b_{\varepsilon'',\hat{T}}((\hat{\ell})_{\sf q})&(\text{by }\eqref{eq:hat})\\
&=&2\left(b_{\varepsilon',\hat{T}}((\hat{\ell})_{\sf p})+b_{\varepsilon'',\hat{T}}((\hat{\ell})_{\sf p})\right)&(\text{by }\eqref{eq:p+q 2})\\
&=&2\left(b_{\varepsilon',\hat{T}}(\widehat{(\ell_{\sf p})})+b_{\varepsilon'',\hat{T}}(\widehat{(\ell_{\sf p})})\right)&(\text{since }(\hat{\ell})_{\sf p}=\widehat{(\ell_{\sf p})})\\
&=&2 b_{\varepsilon,T}(\ell_{\sf p}).&(\text{by }\eqref{eq:shear coordinates})
\end{array}\]
The desired equality \eqref{eq:=p+q} for $\delta\in T\setminus\{\varepsilon\}$ immediately follows from the definition of $(-)_{\sf p}$ and $(-)_{\sf q}$.
\end{proof}

For collections $L$ and $L'$ of laminates of $\cO^w$, we define sets
\[
 \se^{-1}(L) := \{\delta \text{ : a tagged arc}\mid \se(\delta) \in L\}\ \text{ and }\ L \setminus L' := \{\ell \in L \mid \ell \notin L'\}.
\]
The following properties are proved in \cite[Proposition 2.9]{Y20} for the case that $\cO^w$ has no orbifold points. Similarly, we can get the same statement for the general case.

%%%
\begin{proposition}\label{prop:properties of pq}
Let $L$ be a lamination on $\cO^w$ consisting only of elementary and exceptional laminates. Then the following properties hold:
\begin{itemize}
\item[(1)] $C(L) \subseteq C(L_{\sf pq})$.
\item[(2)] $\se^{-1}(L_{\sf pq})$ is a partial tagged triangulation of $\cO^w$.
\end{itemize}
Moreover, we take a partial tagged triangulation $U$ of $\cO^w$ such that $T' = \se^{-1}(L_{\sf pq}) \sqcup U$ is a tagged triangulation. Then we have the equality
\begin{itemize}
 \item[(3)] $C(\se T') = C(L_{\sf pq} \sqcup \se U)$.
\end{itemize}
\end{proposition}

\begin{proof}
 (1) The assertion immediately follows from Lemma \ref{lem:=p+q}.

 (2) Let $L_{\rm el}$ (resp., $L_{\rm ex}$) be the collection of elementary (resp., exceptional) laminates of $L$. By Proposition \ref{prop:properties of elementary}(2), $\se^{-1}(L_{\rm el})$ is a partial tagged triangulation of $\cO^w$. Since $L$ is a lamination, any laminate of $(L_{\rm ex})_{\sf pq}$ is compatible with all laminates of $L_{\rm el} \setminus (L_{\rm ex})_{\sf pq}$. Then, by Proposition \ref{prop:properties of elementary}(1), any tagged arc of $\se^{-1}((L_{\rm ex})_{\sf pq})$ is compatible with all tagged arcs of $\se^{-1}(L_{\rm el} \setminus (L_{\rm ex})_{\sf pq})$. Moreover, $\se^{-1}((L_{\rm ex})_{\sf pq})$ is a partial tagged triangulation since for $\ell \in L_{\rm ex}$, either $\{\se^{-1}(\ell_{\sf p}),\se^{-1}(\ell_{\sf q})\}$ is a pair of conjugate arcs or $\se^{-1}(\ell_{\sf p})=\se^{-1}(\ell_{\sf q})$. Therefore,
\[
 \se^{-1}(L_{\sf pq}) = \se^{-1}(L_{\rm el} \setminus (L_{\rm ex})_{\sf pq}) \sqcup \se^{-1}((L_{\rm ex})_{\sf pq})
\]
 is a partial tagged triangulation of $\cO^w$.

 (3) Since $L_{\sf pq}$ coincides with $\se\se^{-1}(L_{\sf pq})$ up to multiplicity, we have the equalities
\[
C(L_{\sf pq} \sqcup \se U) = C(\se\se^{-1}(L_{\sf pq}) \sqcup \se U) = C(\se T').
\qedhere
\]
\end{proof}

We also interpret the shear coordinates of semi-closed laminates as ones of closed laminates. More generally, we give the following results.

%%%
\begin{proposition}\label{prop:enclosing}
Let $T$ be a tagged triangulation and $\ell$ a laminate of $\cO^w$ with at least one endpoint being an orbifold point with weight $\frac{1}{2}$. For a laminate $\ell'$ enclosing $\ell$ as in Figure \ref{fig:enclosing}, we have
\[
b_T(\ell')=2b_T(\ell).
\]
\end{proposition}

\begin{proof}
If $\ell$ is not semi-closed, then $\ell'$ is an exceptional laminate with $(\ell')_{\sf p}=(\ell')_{\sf q}=\ell$, thus the assertion follows from Lemma \ref{lem:=p+q}. If $\ell$ is semi-closed, then we can get the desired equality in the same way as the proof of Lemma \ref{lem:=p+q}.
\end{proof}

%%%
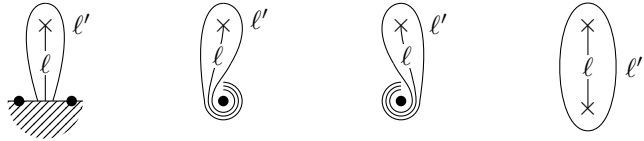
\begin{figure}[htp]
\centering
\begin{tikzpicture}[baseline=0mm]
\coordinate(0)at(0,0);\coordinate(p)at(0,1);
\draw(-0.1,0)..controls(-0.4,1)and(-0.2,1.3)..(0,1.3);
\draw(0.1,0)..controls(0.4,1)and(0.2,1.3)..node[right]{$\ell'$}(0,1.3);
\draw(0)--node[fill=white,inner sep=1]{$\ell$}(p);
\draw(-0.5,0)--(0.5,0);
\fill(-0.35,0)circle(0.07);\fill(0.35,0)circle(0.07);\node at(p){$\times$};
\fill[pattern=north east lines](-0.5,0)arc(-180:0:0.5);
\end{tikzpicture}
 \hspace{10mm}
\begin{tikzpicture}[baseline=0mm]
\coordinate(0)at(0,0);\coordinate(p)at(0,1);
\draw(-0.25,0)..controls(-0.4,1)and(-0.2,1.3)..(0,1.3);
\draw(0.1,0.5)..controls(0.4,1)and(0.2,1.3)..node[right]{$\ell'$}(0,1.3);
\draw(-0.15,0)..controls(-0.15,0.2)and(0,0.3)..(0.1,0.5);
\draw(-0.25,0)arc(-180:90:0.25);
\draw(-0.15,0)arc(-180:90:0.15);
\draw(-0.2,0)arc(-180:90:0.2);\draw(-0.2,0)..controls(-0.2,0.3)and(0,0.7)..node[fill=white,inner sep=1,pos=0.6]{$\ell$}(p);
\fill(0,0)circle(0.07);\node at(p){$\times$};
\end{tikzpicture}
 \hspace{10mm}
\begin{tikzpicture}[baseline=0mm]
\coordinate(0)at(0,0);\coordinate(p)at(0,1);
\draw(0.25,0)..controls(0.4,1)and(0.2,1.3)..node[right]{$\ell'$}(0,1.3);
\draw(-0.1,0.5)..controls(-0.4,1)and(-0.2,1.3)..(0,1.3);
\draw(0.15,0)..controls(0.15,0.2)and(0,0.3)..(-0.1,0.5);
\draw(0:0.25)arc(0:-270:0.25);
\draw(0.15,0)arc(0:-270:0.15);
\draw(0.2,0)arc(0:-270:0.2);\draw(0.2,0)..controls(0.2,0.3)and(0,0.7)..node[fill=white,inner sep=1,pos=0.6]{$\ell$}(p);
\fill(0,0)circle(0.07);\node at(p){$\times$};
\end{tikzpicture}
\hspace{10mm}
\begin{tikzpicture}[baseline=3mm]
\coordinate(0)at(0,0.2);\coordinate(p)at(0,1.3);
\draw(0,-0.1)..controls(-0.5,-0.1)and(-0.5,1.6)..(0,1.6);
\draw(0,-0.1)..controls(0.5,-0.1)and(0.5,1.6)..node[right]{$\ell'$}(0,1.6);
\draw(0)--node[fill=white,inner sep=1]{$\ell$}(p);
\node at(p){$\times$};\node at(0){$\times$};
\end{tikzpicture}
\caption{A laminate $\ell'$ enclosing each laminate $\ell$ with at least one endpoint being an orbifold point with weight $\frac{1}{2}$}\label{fig:enclosing}
\end{figure}

%%%
\begin{corollary}\label{cor:double}
Let $T$ be a tagged triangulation and $L$ a lamination on $\cO^w$. For the set $L'$ of laminates of $L$ with at least one endpoint being an orbifold point with weight $\frac{1}{2}$, we set
\[
L''=\{\ell' \mid \text{$\ell'$ is a laminate enclosing some $\ell\in L'$}\}.
\]
Then, for two distinct sets $L_1=L\setminus L'$ and $L_2=L\setminus L'$, the collection $\overline{L}:=L_1\sqcup L_2\sqcup L''$ is a lamination on $\cO^w$ without semi-closed laminates such that
\[
b_T(\overline{L})=2b_T(L).
\]
\end{corollary}

\begin{proof}
It follows from the construction of $\overline{L}$ that it is a lamination on $\cO^w$ without semi-closed laminates. Then desired equality follows from Proposition \ref{prop:enclosing}.
\end{proof}

Note that the collection $L\sqcup L$ satisfies $b_T(L\sqcup L)=2b_T(L)$. However, it is not a lamination in general because a laminate with at least one endpoint being an orbifold point is not compatible with itself.

%%%
%%%
%%% subsec:Dehn
\subsection{Dehn twists}\label{subsec:Dehn}

Let $V$ be a tubular neighborhood of a closed laminate $\ell_c$, which is homeomorphic to an annulus $S^1 \times [0,1]$. The \emph{Dehn twist $\sT_{\ell_c}$ along $\ell_c$} is a self-homeomorphism of $S$ which is given by sending $(s,t)$ to $(\exp(2\pi i t)s,t)$ on $V \simeq S^1 \times [0,1]$ and by fixing all points on $S \setminus V$ (see e.g. \cite{Ma16}). In particular, $\sT_{\ell_c}$ is oriented as follows:
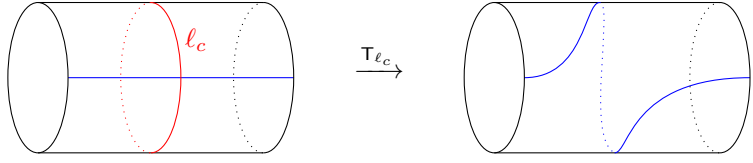
\[
\begin{tikzpicture}[baseline=0mm]
\coordinate(0)at(0,0);
\draw(0,1)--(3,1) (0,-1)--(3,-1);\draw[blue](0.4,0)--(3.4,0);
\draw[red](1.5,1)arc[start angle = 90, end angle = -90, x radius=4mm, y radius=10mm];
\draw[red,dotted](1.5,1)arc[start angle = 90, end angle = 270, x radius=4mm, y radius=10mm];
\draw(0)circle[x radius=4mm, y radius=10mm];
\draw(3,1)arc[start angle = 90, end angle = -90, x radius=4mm, y radius=10mm];
\draw[dotted](3,1)arc[start angle = 90, end angle = 270, x radius=4mm, y radius=10mm];
\node[red]at(2.1,0.5){$\ell_c$};
\end{tikzpicture}
 \hspace{7mm} \xrightarrow{\sT_{\ell_c}} \hspace{7mm}
\begin{tikzpicture}[baseline=0mm]
\coordinate (0) at (0,0);
\draw (0,1)--(3,1) (0,-1)--(3,-1);
\draw[blue] (0.4,0) .. controls (1.2,0) and (1.2,1) .. (1.4,1);
\draw[blue] (1.6,-1) .. controls (1.8,-1) and (1.8,0) .. (3.4,0);
\draw[blue,dotted] (1.4,1) .. controls (1.6,1) and (1.2,-0.8) .. (1.6,-1);
\draw (0) circle [x radius=4mm, y radius=10mm];
\draw (3,1) arc [start angle = 90, end angle = -90, x radius=4mm, y radius=10mm];
\draw[dotted] (3,1) arc [start angle = 90, end angle = 270, x radius=4mm, y radius=10mm];
\end{tikzpicture}
\]\vspace{1mm}\hspace{-1.5mm}
Note that $\sT_{\ell_c}$ is compatible with the map $\widehat{(-)}$ since $\widehat{(-)}$ only affects neighborhoods of marked points, that is,
\begin{equation}\label{eq:compatibility}
\sT_{\ell_c}\widehat{(-)}=\widehat{\sT_{\ell_c}(-)}.
\end{equation}
 
%%% 
\begin{theorem}\label{thm:Dehn}
Let $T$ be a tagged triangulation, $\delta \in T$, $\ell_c$ a closed laminate, and $\ell$ a laminate of $\cO^w$. Then there is $m' \in \bZ_{\ge 0}$ such that for any $m \ge m'$, we have
\[
b_{\delta,T}(\sT_{\ell_c}^m(\ell))=b_{\delta,T}(\sT_{\ell_c}^{m'}(\ell))+(m-m') \#(\ell \cap \ell_c) b_{\delta,T}(\ell_c).
\]
\end{theorem}

\begin{proof}
This is given in \cite[Theorem 2.10]{Y20} for the case that $\cO^w$ has no orbifold points. The general assertion reduces to this case by \eqref{eq:shear coordinates} since $\hat{\cO}^w$ has no orbifold points. In fact, if $\delta$ is a pending arc with weight $\frac{1}{2}$ and $\hat{\delta}=\{\delta',\delta''\}$, then we can take some $m'\in\bZ_{\ge 0}$ such that
\[\arraycolsep=0.5mm\def\arraystretch{1.7}\begin{array}{llll}
b_{\delta,T}(\sT_{\ell_c}^m(\ell))&=&b_{\delta',\hat{T}}(\widehat{\sT_{\ell_c}^m(\ell)})+b_{\delta'',\hat{T}}(\widehat{\sT_{\ell_c}^m(\ell)})&(\text{by }\eqref{eq:shear coordinates})\\
&=&b_{\delta',\hat{T}}(\sT_{\hat{\ell_c}}^m(\hat{\ell}))+b_{\delta'',\hat{T}}(\sT_{\hat{\ell_c}}^m(\hat{\ell}))&(\text{by }\eqref{eq:compatibility})\\
&=&b_{\delta',\hat{T}}(\sT_{\hat{\ell_c}}^{m'}(\hat{\ell}))+b_{\delta'',\hat{T}}(\sT_{\hat{\ell_c}}^{m'}(\hat{\ell}))
+(m-m')\#(\hat{\ell}\cap\hat{\ell_c})(b_{\delta',\hat{T}}(\hat{\ell_c})+b_{\delta'',\hat{T}}(\hat{\ell_c}))\\
&&\multicolumn{2}{r}{(\text{by \cite[Theorem 2.10]{Y20}})}\\
&=&b_{\delta',\hat{T}}(\widehat{\sT_{\ell_c}^{m'}(\ell)})+b_{\delta'',\hat{T}}(\widehat{\sT_{\ell_c}^{m'}(\ell)})
+(m-m')\#(\hat{\ell}\cap\hat{\ell_c})(b_{\delta',\hat{T}}(\hat{\ell_c})+b_{\delta'',\hat{T}}(\hat{\ell_c}))&(\text{by }\eqref{eq:compatibility})\\
&=&b_{\delta,T}(\sT_{\ell_c}^{m'}(\ell))+(m-m')\#(\ell\cap\ell_c)b_{\delta,T}(\ell_c)&(\text{by }\eqref{eq:shear coordinates}).
\end{array}\]

Similarly, we get the desired equality for any $\delta\in T$.
\end{proof}

%%%
\begin{example}\label{ex:dense lam}
We keep the notations of Example \ref{ex:monogon}. For any $m\ge 1$, we have
\[
b_T(\sT_c^m(l_{-1}))=b_T(l_{m-1})=b_T(l_0)+2(m-1)b_T(c)
=b_T(\sT_c(l_{-1}))+(m-1)\#(l_{-1}\cap c)b_T(c).
\]
Thus Theorem \ref{thm:Dehn} holds for $\ell_c=c$ and $\ell=l_{-1}$ by taking $m'=1$. We also have
\[
\lim_{m\rightarrow\infty}\frac{b_T(l_{m-1})}{2(m-1)}=b_T(c),\ \text{ thus }\ b_T(c)\in\overline{\bigcup_{m\ge0}C(\{l_{m}\})}.
\]
On the other hand, all laminations obtained from tagged triangulations of $\cO^w$ by applying $\se$ are given by $\{l_i,r_i\}$ and $\{l_i,r_{i+1}\}$ for $i \in \bZ$. Comparing with Example \ref{ex:monogon}, it is easy to check that the union of their cones $C(\{l_i,r_i\})$ and $C(\{l_i,r_{i+1}\})$ contains $\bZ\setminus C(\{c\})$. The above inclusion means that $C(\{c\})$ is also contained in the union. Therefore, Theorem \ref{thm:dense lam} holds for the weighted orbifold in Example \ref{ex:monogon}.
\end{example}

%%%
%%%
%%% 
\subsection{Proof of the first assertion of Theorem \ref{thm:dense lam}}

In this subsection, we fix a tagged triangulation $T$ and a lamination $L$ on $\cO^w$. By Theorem \ref{thm:thurston}, to prove the first assertion of Theorem \ref{thm:dense lam}, we only need to show that
\begin{equation}\label{eq:vectin}
 b_T(L) \in \overline{\bigcup_{T'}C(\se(T'))},
\end{equation}
where $T'$ runs over all tagged triangulations of $\cO^w$, since the cones are closed under positive real number multiplication. To prove \eqref{eq:vectin}, we need some preparation. We consider a decomposition
\begin{equation*}
L = L_{\rm e} \sqcup L_{\rm sc} \sqcup L_{\rm cl},
\end{equation*}
where $L_{\rm e}$ (resp., $L_{\rm sc}$, $L_{\rm cl}$) consists of all elementary and exceptional (resp., semi-closed, closed) laminates of $L$. By Proposition \ref{prop:properties of pq}(2), $\se^{-1}((L_{\rm e})_{\sf pq})$ is a partial tagged triangulation of $\cO^w$. Then we take a partial tagged triangulation $U$ of $\cO^w$ such that $T_L := \se^{-1}((L_{\rm e})_{\sf pq}) \sqcup U$ is a tagged triangulation.

%%%
\begin{lemma}\label{lem:intersect U}
Any closed laminate $\ell$ of $L_{\rm cl}$ does not intersect with tagged arcs of $\se^{-1}((L_{\rm e})_{\sf pq})$, but it intersects with at least one tagged arc of $U$. In particular, it also intersects with at least one laminate of $\se U$.
\end{lemma}

\begin{proof}
Since $L$ is a lamination, any closed laminate $\ell$ of $L_{\rm cl}$ does not intersect with all laminates of $L_{\rm e}$. Thus $\ell$ does not intersect with all tagged arcs of $\se^{-1}((L_{\rm e})_{\sf pq})$ since $\se^{-1}$ and $(-)_{\sf pq}$ only affect neighborhoods of marked points of $\cO^w$. Moreover, since $T_L$ is a tagged triangulation and $\ell$ is not contractible, $\ell$ intersects with at least one tagged arc $\gamma$ of $T_L$. Since $\gamma$ is contained in $T_L\setminus(\se^{-1}((L_{\rm e})_{\sf pq}))=U$, the first assertion holds. The last assertion immediately follows from the definition of $\se$.
\end{proof}

By Lemma \ref{lem:intersect U}, we have $\sT_{\ell}^m(T_L) = \se^{-1}((L_{\rm e})_{\sf pq}) \sqcup \sT_{\ell}^m(U)$ for any $m\ge0$ and it is a tagged triangulation. Then Proposition \ref{prop:properties of pq}(3) gives
\begin{equation}\label{eq:CDehn}
C(\se\sT_{\ell}^m(T_L)) = C((L_{\rm e})_{\sf pq} \sqcup \se \sT_{\ell}^m(U)).
\end{equation}

Let $\ell_1,\ldots,\ell_t$ be all distinct closed laminates of $L_{\rm cl}$ and $n_i$ the multiplicity of $\ell_i$ in $L_{\rm cl}$ for $i \in \{1,\ldots,t\}$. By Lemma \ref{lem:intersect U}, $N_i := \sum_{\varepsilon \in U}\#(\ell_i\cap\se\varepsilon)$ is not zero. Note that $\#(\ell_i\cap\se\varepsilon)=2\#(\ell_i\cap\varepsilon)$ if $\varepsilon$ is a pending arc with weight $2$. Since $\ell_i$ do not intersect, the Dehn twists $\sT_{\ell_i}$ commute. We set
\[
 \sT := \prod_{i=1}^t \sT_{\ell_i}^{\frac{N_1 \cdots N_t}{N_i}n_i}.
\]
The following is proved in \cite[Proposition 2.14]{Y20} for the case that $\cO^w$ has no orbifold points. Similarly, we can get the similar statement for the general case.

\begin{proposition}\label{prop:inclusion}
Suppose that $L_{\rm sc}=\emptyset$. Then we have
\[
 b_T(L) \in \overline{\bigcup_{m \ge 0} C(\se\sT^m(T_L))}.
\]
\end{proposition}

\begin{proof}
 By Theorem \ref{thm:Dehn}, there is $m' \gg 0$ such that for any $m \ge m'$, we have
{\setlength\arraycolsep{0.5mm}
\begin{eqnarray*}
 b_T(\se\sT^{m}(U)) &=& b_T(\se\sT^{m'}(U)) + \sum_{i=1}^t \Bigl(\frac{N_1 \cdots N_t}{N_i}n_i\Bigr) (m-m') N_i b_T(\ell_i)\\
 &=& b_T(\se\sT^{m'}(U)) + (m-m') N_1 \cdots N_t \sum_{i=1}^t n_i b_T(\ell_i)\\
 &=& b_T(\se\sT^{m'}(U)) + (m-m') N_1 \cdots N_t b_T(L_{\rm cl}).
\end{eqnarray*}}\hspace{-1.5mm}
This equality gives
\[
 \lim_{m \rightarrow \infty} \frac{b_T(\se\sT^m(U))}{m-m'} = N_1 \cdots N_t b_T(L_{\rm cl}),\ \text{ thus }\ b_T(L_{\rm cl}) \in \overline{\bigcup_{m \ge 0} C(\se\sT^m(U))}.
\]
 Since $C(L_{\rm e}) \subseteq C((L_{\rm e})_{\sf pq})$ by Proposition \ref{prop:properties of pq}(1), we have
{\setlength\arraycolsep{0.5mm}
\begin{eqnarray*}
 b_T(L) = b_T(L_{\rm e}) + b_T(L_{\rm cl}) &\in& C(L_{\rm e}) + \overline{ \bigcup_{m \ge 0} C(\se\sT^m(U))}\\
 &\subseteq& \overline{\bigcup_{m \ge 0} C((L_{\rm e})_{\sf pq} \sqcup \se\sT^m(U))}= \overline{\bigcup_{m \ge 0} C(\se\sT^m(T_L))},
\end{eqnarray*}}\hspace{-1.5mm}
where the last equality is given by \eqref{eq:CDehn}.
\end{proof}

\begin{proof}[Proof of the first assertion of Theorem \ref{thm:dense lam}]
By Corollary \ref{cor:double}, $L$ satisfies \eqref{eq:vectin} if and only if $\overline{L}$ is so. Without loss of generality, we can assume that $L_{\rm sc}=\emptyset$ since $\overline{L}_{\rm sc}=\emptyset$.  Since $\sT^m(T_L)$ is a tagged triangulation of $\cO^w$ for any $m \in \bZ_{\ge 0}$, Proposition \ref{prop:inclusion} finishes the proof of \eqref{eq:vectin}. Hence the desired assertion holds.
\end{proof}

%%%
%%%
%%% 
\subsection{Proof of the second assertion of Theorem \ref{thm:dense lam}}

To prove the second assertion of Theorem \ref{thm:dense lam}, we first recall the following result for the case that $\cO^w$ has no orbifold points.

\begin{proposition}[{\cite[Proposition 2.15]{Y20}}]\label{prop:Y20 sum}
Suppose that $\cO^w$ has no orbifold points. Let $T$ be an ideal triangulation without self-folded triangles and $\ell$ a laminate of $\cO^w$ whose both ends are spirals. Then we have
\[
\sum_{\gamma \in T}b_{\gamma,T}(\ell) = \left\{
\begin{array}{ll}
-1 & \text{if both spirals of $\ell$ are clockwise},\\
1 & \text{if both spirals of $\ell$ are counterclockwise},\\
0 & \text{otherwise}.
\end{array} \right.
\]
\end{proposition}

Note that \cite[Proposition 2.15]{Y20} only consider elementary laminates, but it can be given for exceptional laminates in the same way.

Next, we consider a sum as in Proposition \ref{prop:Y20 sum} for the case that $T$ has self-folded triangles. Suppose that $\cO^w$ has no orbifold points. Let $T$ be any ideal triangulation and $\ell$ a laminate of $\cO^w$ whose both ends are spirals. We consider the sum
\[
b_T^{\sum}(\ell):=\sum_{\gamma \in T \setminus T_{\rm s}}b_{\gamma,T}(\ell)+\frac{1}{2}\sum_{\gamma \in T_{\rm s}}b_{\gamma,T}(\ell),
\]
where $T_{\rm s}$ is a set of ideal arcs appearing in self-folded triangles of $T$. To get this sum, we construct a triangulated polygon $T_{\ell}$ associated with $\ell$ as follows (see Figure \ref{fig:Tdelta}): Since $\ell$ is either elementary or exceptional, we can obtain an ideal arc $\gamma_{\ell}$ from $\ell$ by applying the same transformation as $\se^{-1}$ and forgetting its tags. Let $\tau_1,\ldots,\tau_n$ be the arcs of $T$ crossing $\gamma_{\ell}$ in order of occurrence along $\gamma_{\ell}$ (possibly $\tau_{i}=\tau_j$ even if $i \neq j$). Hence $\gamma_{\ell}$ crosses $n+1$ triangles $\triangle_0,\ldots,\triangle_n$ in this order. Suppose first that $T$ has no self-folded triangles. For $i \in \{0,\ldots,n\}$, let $\triangle_{\ell,i}$ be a copy of the oriented triangle $\triangle_i$, hence $\triangle_{\ell,i}$ contains the sides $\tau_i$ and $\tau_{i+1}$ (only $\tau_1$ if $i=0$, and only $\tau_n$ if $i=n$). Then $T_{\ell}'$ is the triangulation of an $(n+3)$-gon obtained by gluing these triangles along the edges $\tau_i$. Similarly, we construct $T_{\ell}$ by adjoining to $T_{\ell}'$ copies of all triangles incident to endpoints of $\gamma_{\ell}$. If $T$ has self-folded triangles, then we adapt the construction using the local transformations as in Figure \ref{fig:local trans}.
%%%
%%%
%%% fig:Tdelta
\begin{figure}[h]
\centering
\begin{tikzpicture}[baseline=-0.5mm]
\coordinate(l)at(0,0);\coordinate(lu)at(0.8,1);\coordinate(ld)at(0.8,-1);
\coordinate(ru)at(3.1,1);\coordinate(rd)at(3.1,-1);\coordinate(rru)at(4.7,1);\coordinate(rrd)at(4.7,-1);
\coordinate(llu)at(-0.8,1);\coordinate(lld)at(-0.8,-1);
\coordinate(r)at(3.9,0);\coordinate(u)at(2.1,1);\coordinate(d)at(1.8,-1);
\node at(2,-0.4) {$\cdots$};
\draw(llu)--(rru)--(rrd)--(lld)--(llu)--(ld)--node[fill=white,inner sep=2,pos=0.3]{$\tau_1$}(lu)--(lld);
\draw(rru)--(rd)--node[fill=white,inner sep=2,pos=0.3]{$\tau_n$}(ru)--(rrd);
\draw[blue](l)--node[above]{$\gamma_{\ell}$}(r);
\draw[loosely dotted](4.4,-0.2)arc(-30:30:0.5);\draw[loosely dotted](-0.5,0.2)arc(150:210:0.5);
\fill(l)circle(0.07);\fill(llu)circle(0.07);\fill(lld)circle(0.07);\fill(lu)circle(0.07);\fill(ld)circle(0.07);
\fill(r)circle(0.07);\fill(rru)circle(0.07);\fill(rrd)circle(0.07);\fill(ru)circle(0.07);\fill(rd)circle(0.07);
\end{tikzpicture}
\caption{A triangulated polygon $T_{\ell}$ for a laminate $\ell$ whose both ends are spirals}\label{fig:Tdelta}
\end{figure}
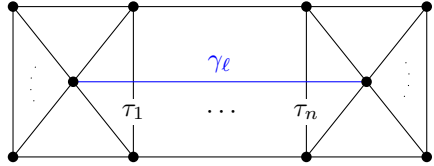
%%%
%%%
%%% fig:replace sf
\begin{figure}[h]
\centering
\begin{tikzpicture}[baseline=7mm]
\coordinate(d)at(0,0);\coordinate(p)at(0,1.3);\node(de)at(-1.3,0.4){$\gamma_{\ell}$};
\draw(d)--(p);
\draw(d)to[out=150,in=180](0,2);\draw(d)to[out=30,in=0](0,2);
\draw[blue,dotted](de)--(0.7,0.4);\draw[blue,dotted](de)--(d);\draw[blue,dotted](de)--(p);
\fill(d)circle(0.07);\fill(p)circle(0.07);
\end{tikzpicture}
\hspace{3mm}$\rightarrow$\hspace{3mm}
\begin{tikzpicture}[baseline=7mm]
\coordinate(d)at(0,0);\coordinate(l)at(-1,1.7);\coordinate(p)at(0,1.2);\coordinate(r)at(1,1.7);\node(de)at(-1.3,0.4){$\gamma_{\ell}$};
\draw(d)--(l)--(r)--(d)--(p) (l)--(p)--(r);
\draw[blue,dotted](de)--(0.7,0.4);\draw[blue,dotted](de)--(d);\draw[blue,dotted](de)--(p);
\fill(d)circle(0.07);\fill(p)circle(0.07);\fill(l)circle(0.07);\fill(r)circle(0.07);
\end{tikzpicture}
\caption{Local transformations around each self-folded triangle}\label{fig:local trans}
\end{figure}
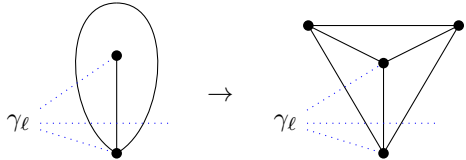
%%%

Since $T_{\ell}$ has no self-folded triangles, the sum $b_{T_{\ell}}^{\sum}(\ell)$ is given by Proposition \ref{prop:Y20 sum}. Using this sum, we give the sum $b_T^{\sum}(\ell)$. For that, we denote by $p$ and $q$ endpoints of $\gamma_{\ell}$ and we say that each of them is \emph{folded} (with respect to $T$) if it is inside a self-folded triangle of $T$. Then we consider the following cases.
\begin{itemize}
\item[(1)] Suppose that $p$ and $q$ are not folded. Since $T$ is an ideal triangulation, each self-folded triangle of $T$ must be in a digon. Thus when $\ell$ crosses over a self-folded triangle of $T$, the corresponding segment of $\ell$ on $T_{\ell}$ is given by
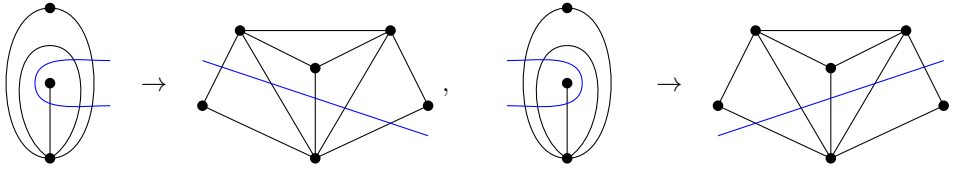
\[
\begin{tikzpicture}[baseline=9mm]
\coordinate(d)at(0,0);\coordinate(p)at(0,1);\coordinate(u)at(0,2);
\draw(d)--(p);
\draw(d)to[out=150,in=180](0,1.5);\draw(d)to[out=30,in=0](0,1.5);
\draw(d)to[out=180,in=180](u);\draw(d)to[out=0,in=0](u);
\draw[blue](0.8,1.3)to[out=180,in=90](-0.2,1);
\draw[blue](-0.2,1)to[out=-90,in=180](0.8,0.7);
\fill(d)circle(0.07);\fill(p)circle(0.07);\fill(u)circle(0.07);
\end{tikzpicture}
\hspace{3mm}\rightarrow\hspace{3mm}
\begin{tikzpicture}[baseline=9mm]
\coordinate(d)at(0,0);\coordinate(l)at(-1,1.7);\coordinate(p)at(0,1.2);\coordinate(r)at(1,1.7);
\coordinate(ll)at(-1.5,0.7);\coordinate(rr)at(1.5,0.7);
\draw(d)--(l)--(r)--(d)--(p) (l)--(p)--(r) (l)--(ll)--(d)--(rr)--(r);
\draw[blue](-1.5,1.3)--(1.5,0.3);
\fill(d)circle(0.07);\fill(p)circle(0.07);\fill(l)circle(0.07);\fill(r)circle(0.07);\fill(ll)circle(0.07);\fill(rr)circle(0.07);
\end{tikzpicture}
\ ,\hspace{7mm}
\begin{tikzpicture}[baseline=9mm]
\coordinate(d)at(0,0);\coordinate(p)at(0,1);\coordinate(u)at(0,2);
\draw(d)--(p);
\draw(d)to[out=150,in=180](0,1.5);\draw(d)to[out=30,in=0](0,1.5);
\draw(d)to[out=180,in=180](u);\draw(d)to[out=0,in=0](u);
\draw[blue](-0.8,1.3)to[out=0,in=90](0.2,1);
\draw[blue](0.2,1)to[out=-90,in=0](-0.8,0.7);
\fill(d)circle(0.07);\fill(p)circle(0.07);\fill(u)circle(0.07);
\end{tikzpicture}
\hspace{3mm}\rightarrow\hspace{3mm}
\begin{tikzpicture}[baseline=9mm]
\coordinate(d)at(0,0);\coordinate(l)at(-1,1.7);\coordinate(p)at(0,1.2);\coordinate(r)at(1,1.7);
\coordinate(ll)at(-1.5,0.7);\coordinate(rr)at(1.5,0.7);
\draw(d)--(l)--(r)--(d)--(p) (l)--(p)--(r) (l)--(ll)--(d)--(rr)--(r);
\draw[blue](-1.5,0.3)--(1.5,1.3);
\fill(d)circle(0.07);\fill(p)circle(0.07);\fill(l)circle(0.07);\fill(r)circle(0.07);\fill(ll)circle(0.07);\fill(rr)circle(0.07);
\end{tikzpicture}\ .
\]
Then it is easy to check that $b_T^{\sum}(\ell)=b_{T_{\ell}}^{\sum}(\ell)$ (see also Example \ref{ex:digon}).
\item[(2)] Suppose that $p$ is not folded and $q$ is folded. If the spiral of $\ell$ at $q$ is clockwise, then the corresponding end of $\ell$ at $q$ on $T_{\ell}$ is given by
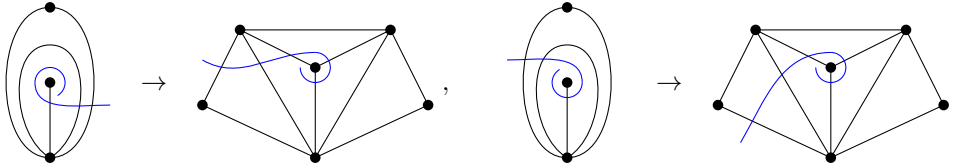
\[
\begin{tikzpicture}[baseline=9mm]
\coordinate(d)at(0,0);\coordinate(p)at(0,1);\coordinate(u)at(0,2);
\draw(d)--(p);
\draw(d)to[out=150,in=180](0,1.5);\draw(d)to[out=30,in=0](0,1.5);
\draw(d)to[out=180,in=180](u);\draw(d)to[out=0,in=0](u);
\draw[blue](-0.2,1)to[out=-90,in=180](0.8,0.7);
\draw[blue](-0.2,1)arc(180:-60:0.2);
\fill(d)circle(0.07);\fill(p)circle(0.07);\fill(u)circle(0.07);
\end{tikzpicture}
\hspace{3mm}\rightarrow\hspace{3mm}
\begin{tikzpicture}[baseline=9mm]
\coordinate(d)at(0,0);\coordinate(l)at(-1,1.7);\coordinate(p)at(0,1.2);\coordinate(r)at(1,1.7);
\coordinate(ll)at(-1.5,0.7);\coordinate(rr)at(1.5,0.7);
\draw(d)--(l)--(r)--(d)--(p) (l)--(p)--(r) (l)--(ll)--(d)--(rr)--(r);
\draw[blue](-1.5,1.3)to[out=-30,in=180](0,1.4);
\draw[blue](0,1.4)arc(90:-180:0.2);
\fill(d)circle(0.07);\fill(p)circle(0.07);\fill(l)circle(0.07);\fill(r)circle(0.07);\fill(ll)circle(0.07);\fill(rr)circle(0.07);
\end{tikzpicture}
\ ,\hspace{7mm}
\begin{tikzpicture}[baseline=9mm]
\coordinate(d)at(0,0);\coordinate(p)at(0,1);\coordinate(u)at(0,2);
\draw(d)--(p);
\draw(d)to[out=150,in=180](0,1.5);\draw(d)to[out=30,in=0](0,1.5);
\draw(d)to[out=180,in=180](u);\draw(d)to[out=0,in=0](u);
\draw[blue](-0.8,1.3)to[out=0,in=90](0.2,1);
\draw[blue](0.2,1)arc(0:-240:0.2);
\fill(d)circle(0.07);\fill(p)circle(0.07);\fill(u)circle(0.07);
\end{tikzpicture}
\hspace{3mm}\rightarrow\hspace{3mm}
\begin{tikzpicture}[baseline=9mm]
\coordinate(d)at(0,0);\coordinate(l)at(-1,1.7);\coordinate(p)at(0,1.2);\coordinate(r)at(1,1.7);
\coordinate(ll)at(-1.5,0.7);\coordinate(rr)at(1.5,0.7);
\draw(d)--(l)--(r)--(d)--(p) (l)--(p)--(r) (l)--(ll)--(d)--(rr)--(r);
\draw[blue](-1.2,0.2)to[out=60,in=180](0,1.4);
\draw[blue](0,1.4)arc(90:-180:0.2);
\fill(d)circle(0.07);\fill(p)circle(0.07);\fill(l)circle(0.07);\fill(r)circle(0.07);\fill(ll)circle(0.07);\fill(rr)circle(0.07);
\end{tikzpicture}\ .
\]
Combined with (1), it is easy to check that $b_T^{\sum}(\ell)=b_{T_{\ell}}^{\sum}(\ell)+\frac{1}{2}$. Similarly, if the spiral of $\ell$ at $q$ is counterclockwise, then we have $b_T^{\sum}(\ell)=b_{T_{\ell}}^{\sum}(\ell)-\frac{1}{2}$ (see also Example \ref{ex:digon}). In particular, by Proposition \ref{prop:Y20 sum}, we have
\[
b_T^{\sum}(\ell)=\left\{
\begin{array}{ll}
-\frac{1}{2} & \text{if the spiral of $\ell$ at $p$ is clockwise},\\
\frac{1}{2} & \text{if the spiral of $\ell$ at $p$ is counterclockwise}.
\end{array} \right.
\]
\item[(3)] Suppose that $p$ and $q$ are folded. In the same way as (2), we get $b_T^{\sum}(\ell)=0$.
\end{itemize}

Finally, we consider the case that $\cO^w$ has orbifold points. Let $T$ be any ideal triangulation and $\ell$ a laminate of $\cO^w$ with each end being either a spiral or an orbifold point with weight $\frac{1}{2}$. Then we consider the sum
\[
b_T^{\sum}(\ell):=\sum_{\gamma \in T \setminus T_{\rm ps}}b_{\gamma,T}(\ell)+\frac{1}{2}\sum_{\gamma \in T_{\rm ps}}b_{\gamma,T}(\ell),
\]
where $T_{\rm ps}$ is the set of pending arcs of $T$ with weight $\frac{1}{2}$ and ideal arcs appearing in self-folded triangles of $T$. By \eqref{eq:shear coordinates}, we have $b_T^{\sum}(\ell)=b_{\hat{T}}^{\sum}(\hat{\ell})$. Thus it is given by the above observation since $\hat{\cO}^w$ has no orbifold points.

\begin{proof}[Proof of the second assertion of Theorem \ref{thm:dense lam}]
Suppose that $\cO^w$ has empty boundary and exactly one puncture $p$. Let $T$ be a tagged triangulation and $\ell$ an elementary or exceptional laminate of $\cO^w$. Without loss of generality, we can assume that $T$ is an ideal triangulation without self-folded triangles. If both ends of $\ell$ are spirals at $p$, then the above observation gives
\[
b_T^{\sum}(\ell)=b_{\hat{T}}^{\sum}(\hat{\ell})=\left\{
\begin{array}{ll}
-1 & \text{if both spirals of $\ell$ at $p$ are clockwise},\\
1 & \text{if both spirals of $\ell$ at $p$ are counterclockwise}.
\end{array} \right.
\]
If an end of $\ell$ is an orbifold point with weight $\frac{1}{2}$, then the above observation also gives
\[
b_T^{\sum}(\ell)=b_{\hat{T}}^{\sum}(\hat{\ell})=\left\{
\begin{array}{ll}
-\frac{1}{2} & \text{if a spiral of $\ell$ at $p$ is clockwise},\\
\frac{1}{2} & \text{if a spiral of $\ell$ at $p$ is counterclockwise}.
\end{array} \right.
\]
In particular, in both cases, we have
\[
b_T^{\sum}(\ell)\left\{
\begin{array}{ll}
<0 & \text{if spirals of $\ell$ at $p$ are clockwise},\\
>0 & \text{if spirals of $\ell$ at $p$ are counterclockwise}.
\end{array} \right.
\]
Therefore, for a tagged arc $\delta$ of $\cO^w$, we also have
\[
b_T^{\sum}(\se(\delta))\left\{
\begin{array}{ll}
<0 & \text{if tags of $\delta$ at $p$ are plain},\\
>0 & \text{if tags of $\delta$ at $p$ are notched}.
\end{array} \right.
\]
Thus the desired assertion follows from the first assertion of Theorem \ref{thm:dense lam}.
\end{proof}

%%%
%%%
%%%
\section{Denseness of $g$-vector cones from weighted orbifolds}\label{sec:cluster algebra}
\subsection{Cluster algebras}

We recall the definition of cluster algebras for a special case, that is, cluster algebras with principal coefficients \cite{FZ07}. We only study this case and refer to \cite{FZ02,FZ07} for a general definition of cluster algebras. Before giving the definition, we prepare some notations.

Let $m \ge n$ be positive integers and $A=(a_{ij})$ an $m\times n$ integer matrix whose upper part $(a_{ij})_{1 \le i, j \le n}$ is skew-symmetrizable, that is, there are positive numbers $d_1,\ldots,d_n$ such that $a_{ij}d_j=-a_{ji}d_i$ for $1 \le i,j \le n$. For $1 \le k \le n$, the \emph{mutation of $A$ at $k$} is the matrix $\mu_k A=(a'_{ij})$ given by
\[
a'_{ij} =  \left\{\begin{array}{ll}
 -a_{ij} & \mbox{if} \ \ i=k \ \ \mbox{or} \ \ j=k,\\
 a_{ij}+a_{ik}[a_{kj}]_+ + [-a_{ik}]_+a_{kj} & \mbox{otherwise},
\end{array} \right.
\]
where $[a]_+:={\rm max}(a,0)$. It is easy to check that the upper part of $\mu_k A$ is still skew-symmetrizable. Moreover, $\mu_k$ is an involution, that is, $\mu_k \mu_kA=A$.

%%%
\begin{example}\label{ex:matrix}
For the following $4\times 2$ matrix whose upper part is skew-symmetrizable, a sequence of its mutations is given by
\[
\begin{bmatrix}0&-1\\4&0\\1&0\\0&1\end{bmatrix}\xleftrightarrow{\mu_2}
\begin{bmatrix}0&1\\-4&0\\1&0\\4&-1\end{bmatrix}\xleftrightarrow{\mu_1}
\begin{bmatrix}0&-1\\4&0\\-1&1\\-4&3\end{bmatrix}\xleftrightarrow{\mu_2}
\begin{bmatrix}0&1\\-4&0\\3&-1\\8&-3\end{bmatrix}.
\]
\end{example}

Let $\cF$ be the field of rational functions in $2n$ variables over $\bQ$. A \emph{seed} (with coefficients) is a pair $(\bx,A)$ consisting of the following data:
\begin{itemize}
\item $\bx=(x_1,\ldots,x_n,y_1,\ldots,y_n)$ is a free generating set of $\cF$ over $\bQ$;
\item $A=(a_{ij})$ is a $2n \times n$ integer matrix whose upper part is skew-symmetrizable.
\end{itemize}
Then we refer to the tuple $(x_1,\ldots,x_n)$ as the \emph{cluster}, each $x_i$ as a \emph{cluster variable}, and $y_i$ as a \emph{coefficient}. For $1 \le k \le n$, the \emph{mutation of $(\bx,A)$ at $k$} is the seed $\mu_k(\bx,A)=((x'_1,\ldots,x'_n,y_1,\ldots,y_n),\mu_kA)$ defined by $x'_i = x_i$ if $i \neq k$, and
\[
x_k x'_k = \prod_{i=1}^nx_i^{[a_{ik}]_+}\prod_{i=1}^ny_i^{[a_{n+i,k}]_+}+\prod_{i=1}^nx_i^{[-a_{ik}]_+}\prod_{i=1}^ny_i^{[-a_{n+i,k}]_+}.
\]
It is easy to check that $\mu_k(\bx,A)$ is also a seed and $\mu_k$ is an involution. Moreover, the coefficients are not changed by mutations.

Let $B$ be an $n \times n$ skew-symmetrizable integer matrix. The \emph{principal extension of $B$} is the $2n \times n$ skew-symmetrizable matrix $\hat{B}$ whose upper part is $B$ and whose lower part is the identity matrix of rank $n$. We fix a seed $(\bx=(x_1,\ldots,x_n,y_1,\ldots,y_n),\hat{B})$, called the \emph{initial seed}. We also call the tuple $(x_1,\ldots,x_n)$ the \emph{initial cluster}, and each $x_i$ the \emph{initial cluster variable}. The \emph{cluster algebra $\cA(B)=\cA(\bx,\hat{B})$ with principal coefficients} for the initial seed $(\bx,\hat{B})$ is a $\bZ$-subalgebra of $\cF$ generated by the cluster variables and the coefficients obtained from $(\bx,\hat{B})$ by all sequences of mutations. We denote by $\cluster B$ the set of clusters of $\cA(B)$ and by $\clv B$ the set of cluster variables of $\cA(B)$.

One of the remarkable properties of cluster algebras with principal coefficients is the (strongly) \emph{Laurent phenomenon}.

\begin{theorem}[\cite{FZ02,FZ07}]\label{thm:Laurent phenomenon}
Any cluster variable $x$ is expressed by a Laurent polynomial of the initial cluster variables $x_1,\ldots,x_n$ and coefficients $y_1,\ldots,y_n$ of the form
\[
x=\frac{f(x_1,\ldots,x_n,y_1,\ldots,y_n)}{x_1^{d_1} \cdots x_n^{d_n}},
\]
where $f(x_1,\ldots,x_n,y_1,\ldots,y_n) \in \bZ[x_1,\ldots,x_n,y_1,\ldots,y_n]$ and $d_i \in \bZ_{\ge 0}$.
\end{theorem}

Theorem \ref{thm:Laurent phenomenon} means that $\cA(B)$ is contained in $\bZ[x_1^{\pm 1},\ldots,x_n^{\pm 1},y_1,\ldots,y_n]$. We consider its $\bZ^n$-grading given by
\[
 \deg(x_i)=\be_i,\ \ \deg(y_j)=\sum_{i=1}^n -b_{ij}\be_i,
\]
where $\be_1,\ldots,\be_n$ are the standard basis vectors in $\bZ^n$. Then every cluster variable $x$ of $\cA(B)$ is homogeneous with respect to this $\bZ^n$-grading \cite[Proposition 6.1]{FZ07}. The degree $g(x):=\deg(x)$ is called the \emph{$g$-vector} of $x$. The cone spanned by the $g$-vectors of cluster variables in a cluster $\bx$ is called the \emph{$g$-vector cone} of $\bx$, that is, it is given by $C(\bx)=\{\sum_{x \in \bx}a_x g(x)\mid a_x\in\bR_{\ge 0}\}$.

%%%
\begin{example}\label{ex:cluster}
Let $B=\begin{bmatrix}0&-1\\4&0\end{bmatrix}$. Then $\hat{B}$ is the matrix in Example \ref{ex:matrix}. Applying the same sequence of mutations as Example \ref{ex:matrix} to the initial seed $((x_1,x_2,y_1,y_2),\hat{B})$, the corresponding clusters of $\cA(B)$ are given by
{\setlength\arraycolsep{0.5mm}
\begin{eqnarray*}
\left(x_1,x_2\right)\xleftrightarrow{\mu_2}
\left(x_1,\frac{x_1+y_2}{x_2}\right)&\xleftrightarrow{\mu_1}&
\left(\frac{(x_1+y_2)^4+y_1y_2^4x_2^4}{x_1x_2^4},\frac{x_1+y_2}{x_2}\right)\\
&\xleftrightarrow{\mu_2}&
\left(\frac{(x_1+y_2)^4+y_1y_2^4x_2^4}{x_1x_2^4},\frac{(x_1+y_2)^3+y_1y_2^3x_2^4}{x_1x_2^3}\right).
\end{eqnarray*}}\hspace{-1.5mm}
Then the degrees $\deg(x_1)=\deg(y_2)=(1,0), \deg(x_2)=(0,1), \deg(y_1)=(0,-4)$
gives the $g$-vectors
\[
g\Bigl(\frac{x_1+y_2}{x_2}\Bigr)=(1,-1),\ g\Bigl(\frac{(x_1+y_2)^4+y_1y_2^4x_2^4}{x_1x_2^4}\Bigr)=(3,-4),\ g\Bigl(\frac{(x_1+y_2)^3+y_1y_2^3x_2^4}{x_1x_2^3}\Bigr)=(2,-3).
\]
\end{example}

%%%
%%%
%%%
\subsection{Cluster algebras associated with weighted orbifolds}

Let $\cO^w=(S,M,Q)$ be a weighted orbifold. To a tagged triangulation $T$ of $\cO^w$, we associate a $|T|\times|T|$ skew-symmetrizable integer matrix $B_T$ as follows (see \cite[Tables 3.3, 3.4, and 3.5]{FeST12a}): For each triangle $\triangle$ of $T$, we define the $|T|\times|T|$ integer matrix $B^{\triangle}=(b_{ij}^{\triangle})$, where $b_{ij}^{\triangle}$ is given as in Figures \ref{fig:matrices non-orbifold}, \ref{fig:matrices orbifold} and \ref{Fig:matrices-exceptional} if $i$ and $j$ are sides of $\triangle$ and $b_{ij}^{\triangle}=0$ otherwise. Then we define the matrix
\[
B_T:=\sum_{\triangle}B^{\triangle},
\]
where $\triangle$ runs over all triangles of $T$. For $i \in T$, we take $d_i$ as a weight of $i$, where $d_i=1$ if $i$ is not a pending arc. Then $B_T=(b_{ij})_{i,j\in T}$ is skew-symmetrizable such that $b_{ij}d_j=-b_{ji}d_i$.

%%%
%%%
%%% fig:matrices non-orbifold
\begin{figure}[htp]
\centering
\begin{tabular}{c|c|c|c}
$\triangle$ &
\begin{tikzpicture}[baseline=7]
\coordinate(u)at(0,1);\coordinate(l)at(-0.7,-0.1);\coordinate(r)at(0.7,-0.1);
\draw(u)--node[above left]{$1$}(l)--node[below]{$3$}(r)--node[above right]{$2$}(u);
\fill(u)circle(0.07);\fill(l)circle(0.07);\fill(r)circle(0.07);\node at(0,1.2){};
\end{tikzpicture}
 &
\begin{tikzpicture}[baseline=-4]
\coordinate(p)at(0,0);\coordinate(u)at(0,1);\coordinate(d)at(0,-1);
\draw(d)to[out=180,in=180]node[left]{$1$}(u);\draw(d)to[out=0,in=0]node[right]{$2$}(u);
\draw(d)--node[left]{$3$}(p);
\draw(p)to[out=80,in=100,relative]node[pos=0.2]{\rotatebox{40}{\footnotesize $\bowtie$}}node[fill=white,inner sep=1.5]{$4$}(d);
\fill(p)circle(0.07);\fill(u)circle(0.07);\fill(d)circle(0.07);\node at(0,1.1){};\node at(0,-1.1){};
\end{tikzpicture}
 &
\begin{tikzpicture}[baseline=-4]
\coordinate(d)at(0,-1);\coordinate(l)at(-0.4,0);\coordinate(r)at(0.4,0);
\draw(0,0)circle(1);\draw(l)--node[fill=white,inner sep=1.5]{$2$}(d)--node[fill=white,inner sep=1.5]{$4$}(r);
\draw(l)to[out=-80,in=-100,relative]node[pos=0.2]{\rotatebox{-25}{\footnotesize $\bowtie$}}node[fill=white,inner sep=1.5,pos=0.45]{$3$}(d);
\draw(r)to[out=80,in=100,relative]node[pos=0.2]{\rotatebox{25}{\footnotesize $\bowtie$}}node[fill=white,inner sep=1.5,pos=0.45]{$5$}(d);
\fill(d)circle(0.07);\fill(l)circle(0.07);\fill(r)circle(0.07);\node at(0,1.2){};\node at(0,-1.2){};
\node at(0,0.75){$1$};
\end{tikzpicture}
\\\hline
 $(b_{ij}^{\triangle})_{i,j \in \triangle}$ & $\begin{bmatrix}0&1&-1\\-1&0&1\\1&-1&0\end{bmatrix}$
 & $\begin{bmatrix}0&1&-1&-1\\-1&0&1&1\\1&-1&0&0\\1&-1&0&0\end{bmatrix}$
 & $\begin{bmatrix}0&-1&-1&1&1\\1&0&0&-1&-1\\1&0&0&-1&-1\\-1&1&1&0&0\\-1&1&1&0&0\end{bmatrix}$
\begin{tikzpicture}[baseline=-4]\node at(0,1.1){};\node at(0,-1.1){};\end{tikzpicture}%control space
\end{tabular}
\caption{Matrices associated with non-exceptional triangles without orbifold points}
\label{fig:matrices non-orbifold}
\end{figure}

%%%
%%%
%%% fig:matrices orbifold
\begin{figure}[htp]
\centering
\begin{tabular}{c|c|c}
$\triangle$ & Weights & $(b_{ij}^{\triangle})_{i,j \in \triangle}$
\\\hline
\multirow{2}{*}{
\begin{tikzpicture}[baseline=0]
\coordinate(p)at(0,0);\coordinate(u)at(0,1);\coordinate(d)at(0,-1);
\draw(d)to[out=180,in=180]node[left]{$1$}(u);\draw(d)to[out=0,in=0]node[right]{$2$}(u);
\draw(d)--node[left]{$3$}(p)node[above]{$p$};
\node at(p){$\times$};\fill(u)circle(0.07);\fill(d)circle(0.07);
\end{tikzpicture}
} \hspace{1mm}
 & $w(p)=\frac{1}{2}$ & $\begin{bmatrix}0&1&-2\\-1&0&2\\1&-1&0\end{bmatrix}$
\begin{tikzpicture}[baseline=-4]\node at(0,0.8){};\node at(0,-0.8){};\end{tikzpicture}%control space
\\\cline{2-3}%
 & $w(p)=2$ & $\begin{bmatrix}0&1&-1\\-1&0&1\\2&-2&0\end{bmatrix}$
\begin{tikzpicture}[baseline=-4]\node at(0,0.8){};\node at(0,-0.8){};\end{tikzpicture}%control space
\\\hline%with puncture and orbifold point
\multirow{2}{*}{
\begin{tikzpicture}[baseline=0]
\coordinate(d)at(0,-1);\coordinate(l)at(-0.4,0);\coordinate(r)at(0.4,0);
\draw(0,0)circle(1);\draw(l)--node[fill=white,inner sep=1.5]{$2$}(d)--node[fill=white,inner sep=1.5]{$4$}(r)node[above]{$p$};
\draw(l)to[out=-80,in=-100,relative]node[pos=0.2]{\rotatebox{-25}{\footnotesize $\bowtie$}}node[fill=white,inner sep=1.5,pos=0.45]{$3$}(d);
\fill(d)circle(0.07);\fill(l)circle(0.07);\node at(r){$\times$};
\node at(0,0.75){$1$};
\end{tikzpicture}
} \hspace{1mm}
 & $w(p)=\frac{1}{2}$ & $\begin{bmatrix}0&-1&-1&2\\1&0&0&-2\\1&0&0&-2\\-1&1&1&0\end{bmatrix}$
\begin{tikzpicture}[baseline=-4]\node at(0,1){};\node at(0,-1){};\end{tikzpicture}%control space
\\\cline{2-3}%
 & $w(p)=2$ & $\begin{bmatrix}0&-1&-1&1\\1&0&0&-1\\1&0&0&-1\\-2&2&2&0\end{bmatrix}$
\begin{tikzpicture}[baseline=-4]\node at(0,1){};\node at(0,-1){};\end{tikzpicture}%control space
\\\hline%%%
\multirow{2}{*}{
\begin{tikzpicture}[baseline=0]
\coordinate(d)at(0,-1);\coordinate(l)at(-0.4,0);\coordinate(r)at(0.4,0);
\draw(0,0)circle(1);\draw(l)node[above]{$p$}--node[fill=white,inner sep=1.5]{$2$}(d)--node[fill=white,inner sep=1.5]{$3$}(r);
\draw(r)to[out=80,in=100,relative]node[pos=0.2]{\rotatebox{25}{\footnotesize $\bowtie$}}node[fill=white,inner sep=1.5,pos=0.45]{$4$}(d);
\fill(d)circle(0.07);\node at(l){$\times$};\fill(r)circle(0.07);
\node at(0,0.75){$1$};
\end{tikzpicture}
} \hspace{1mm}
 & $w(p)=\frac{1}{2}$ & $\begin{bmatrix}0&-2&1&1\\1&0&-1&-1\\-1&2&0&0\\-1&2&0&0\end{bmatrix}$
\begin{tikzpicture}[baseline=-4]\node at(0,1){};\node at(0,-1){};\end{tikzpicture}%control space
\\\cline{2-3}%
 & $w(p)=2$ & $\begin{bmatrix}0&-1&1&1\\2&0&-2&-2\\-1&1&0&0\\-1&1&0&0\end{bmatrix}$
\begin{tikzpicture}[baseline=-4]\node at(0,1){};\node at(0,-1){};\end{tikzpicture}%control space
\\\hline%with two orbifold points
\multirow{4}{*}{
\begin{tikzpicture}[baseline=0]
\coordinate(d)at(0,-1);\coordinate(l)at(-0.4,0);\coordinate(r)at(0.4,0);
\draw(0,0)circle(1);\draw(l)node[above]{$p$}--node[fill=white,inner sep=1.5]{$2$}(d)--node[fill=white,inner sep=1.5]{$3$}(r)node[above]{$q$};
\fill(d)circle(0.07);\node at(l){$\times$};\node at(r){$\times$};\node at(0,3){};
\node at(0,0.75){$1$};
\end{tikzpicture}
} \hspace{1mm}
 & 
\begin{tikzpicture}[baseline=-4]\node at(0,0.3){$w(p)=\frac{1}{2}$};\node at(0,-0.3){$w(q)=\frac{1}{2}$};\end{tikzpicture}
 & $\begin{bmatrix}0&-2&2\\1&0&-2\\-1&2&0\end{bmatrix}$
\begin{tikzpicture}[baseline=-4]\node at(0,0.8){};\node at(0,-0.8){};\end{tikzpicture}%control space
\\\cline{2-3}%
 & 
\begin{tikzpicture}[baseline=-4]\node at(0,0.3){$w(p)=\frac{1}{2}$};\node at(0,-0.3){$w(q)=2$};\end{tikzpicture}
 & $\begin{bmatrix}0&-2&1\\1&0&-1\\-2&4&0\end{bmatrix}$
\begin{tikzpicture}[baseline=-4]\node at(0,0.8){};\node at(0,-0.8){};\end{tikzpicture}%control space
\\\cline{2-3}%
 & 
\begin{tikzpicture}[baseline=-4]\node at(0,0.3){$w(p)=2$};\node at(0,-0.3){$w(q)=\frac{1}{2}$};\end{tikzpicture}
 & $\begin{bmatrix}0&-1&2\\2&0&-4\\-1&1&0\end{bmatrix}$
\begin{tikzpicture}[baseline=-4]\node at(0,0.8){};\node at(0,-0.8){};\end{tikzpicture}%control space
\\\cline{2-3}%
 & 
\begin{tikzpicture}[baseline=-4]\node at(0,0.3){$w(p)=2$};\node at(0,-0.3){$w(q)=2$};\end{tikzpicture}
 & $\begin{bmatrix}0&-1&1\\2&0&-1\\-2&1&0\end{bmatrix}$
\begin{tikzpicture}[baseline=-4]\node at(0,0.8){};\node at(0,-0.8){};\end{tikzpicture}%control space
\end{tabular}
\caption{Matrices associated with non-exceptional triangles with orbifold points}
\label{fig:matrices orbifold}
\end{figure}

%%%
%%%
%%% fig:matrices exceptional
\begin{figure}[htp]
\centering
\begin{tabular}{c|c|c}
$\triangle$ & Weights & $(b_{ij}^{\triangle})_{i,j \in \triangle}$
\\\hline
\begin{tikzpicture}[baseline=0mm]
\coordinate(0)at(0,0);\coordinate(u)at(90:1);\coordinate(r)at(-30:1);\coordinate(l)at(210:1);
\draw(0)--node[left]{$1$}(u);\draw(0)--node[below]{$3$}(l);\draw(0)--node[above right]{$5$}(r);
\draw(0)to[out=-60,in=-120,relative]node[pos=0.8]{\rotatebox{20}{\footnotesize $\bowtie$}}node[right]{$2$}(u);
\draw(0)to[out=-60,in=-120,relative]node[pos=0.8]{\rotatebox{-30}{\footnotesize $\bowtie$}}node[above]{$4$} (l);
\draw(0)to[out=-60,in=-120,relative]node[pos=0.8]{\rotatebox{100}{\footnotesize $\bowtie$}}node[below]{$6$} (r);
\fill(0)circle(0.07);\fill(u)circle(0.07);\fill(l)circle(0.07);\fill(r)circle(0.07);
\end{tikzpicture}
\hspace{1mm}
 & & $\begin{bmatrix}0&0&1&1&-1&-1\\0&0&1&1&-1&-1\\-1&-1&0&0&1&1\\-1&-1&0&0&1&1\\1&1&-1&-1&0&0\\1&1&-1&-1&0&0\end{bmatrix}$
\begin{tikzpicture}[baseline=-4]\node at(0,1.3){};\node at(0,-1.3){};\end{tikzpicture}%control space
\\\hline%with one orbifold point
\multirow{2}{*}{
\begin{tikzpicture}[baseline=0mm]
\coordinate(0)at(0,0);\coordinate(u)at(90:1);\coordinate(r)at(-30:1);\coordinate(l)at(210:1);
\draw(0)--node[left]{$1$}(u);\draw(0)--node[below]{$3$}(l);\draw(0)--node[below]{$5$}(r)node[below]{$r$};
\draw(0)to[out=-60,in=-120,relative]node[pos=0.8]{\rotatebox{20}{\footnotesize $\bowtie$}}node[right]{$2$}(u);
\draw(0)to[out=-60,in=-120,relative]node[pos=0.8]{\rotatebox{-30}{\footnotesize $\bowtie$}}node[above]{$4$} (l);
\fill(0)circle(0.07);\fill(u)circle(0.07);\fill(l)circle(0.07);\node at(r){$\times$};
\end{tikzpicture}
} \hspace{1mm}
 & $w(r)=\frac{1}{2}$ & $\begin{bmatrix}0&0&1&1&-2\\0&0&1&1&-2\\-1&-1&0&0&2\\-1&-1&0&0&2\\1&1&-1&-1&0\end{bmatrix}$
\begin{tikzpicture}[baseline=-4]\node at(0,1.1){};\node at(0,-1.1){};\end{tikzpicture}%control space
\\\cline{2-3}%
 & $w(r)=2$ & $\begin{bmatrix}0&0&1&1&-1\\0&0&1&1&-1\\-1&-1&0&0&1\\-1&-1&0&0&1\\2&2&-2&-2&0\end{bmatrix}$
\begin{tikzpicture}[baseline=-4]\node at(0,1.1){};\node at(0,-1.1){};\end{tikzpicture}%control space
\\\hline%with two orbifold points
\multirow{4}{*}{
\begin{tikzpicture}[baseline=0mm]
\coordinate(0)at(0,0);\coordinate(u)at(90:1);\coordinate(r)at(-30:1);\coordinate(l)at(210:1);
\draw(0)--node[left]{$1$}(u);\draw(0)--node[below]{$3$}(l)node[below]{$q$};\draw(0)--node[below]{$4$}(r)node[below]{$r$};
\draw(0)to[out=-60,in=-120,relative]node[pos=0.8]{\rotatebox{20}{\footnotesize $\bowtie$}}node[right]{$2$}(u);
\fill(0)circle(0.07);\fill(u)circle(0.07);\node at(l){$\times$};\node at(r){$\times$};\node at(0,3){};
\end{tikzpicture}
} \hspace{1mm}
 &
\begin{tikzpicture}[baseline=-4]\node at(0,0.3){$w(q)=\frac{1}{2}$};\node at(0,-0.3){$w(r)=\frac{1}{2}$};\end{tikzpicture}
 & $\begin{bmatrix}0&0&2&-2\\0&0&2&-2\\-1&-1&0&2\\1&1&-2&0\end{bmatrix}$
\begin{tikzpicture}[baseline=-4]\node at(0,0.9){};\node at(0,-0.9){};\end{tikzpicture}%control space
\\\cline{2-3}%
 &
\begin{tikzpicture}[baseline=-4]\node at(0,0.3){$w(q)=\frac{1}{2}$};\node at(0,-0.3){$w(r)=2$};\end{tikzpicture}
 & $\begin{bmatrix}0&0&2&-1\\0&0&2&-1\\-1&-1&0&1\\2&2&-4&0\end{bmatrix}$
\begin{tikzpicture}[baseline=-4]\node at(0,0.9){};\node at(0,-0.9){};\end{tikzpicture}%control space
\\\cline{2-3}%
 &
\begin{tikzpicture}[baseline=-4]\node at(0,0.3){$w(q)=2$};\node at(0,-0.3){$w(r)=\frac{1}{2}$};\end{tikzpicture}
 & $\begin{bmatrix}0&0&1&-2\\0&0&1&-2\\-2&-2&0&4\\1&1&-1&0\end{bmatrix}$
\begin{tikzpicture}[baseline=-4]\node at(0,0.9){};\node at(0,-0.9){};\end{tikzpicture}%control space
\\\cline{2-3}%
 &
\begin{tikzpicture}[baseline=-4]\node at(0,0.3){$w(q)=2$};\node at(0,-0.3){$w(r)=2$};\end{tikzpicture}
 & $\begin{bmatrix}0&0&1&-1\\0&0&1&-1\\-2&-2&0&1\\2&2&-1&0\end{bmatrix}$
\begin{tikzpicture}[baseline=-4]\node at(0,0.9){};\node at(0,-0.9){};\end{tikzpicture}%control space
\\\hline%with three orbifold points
\multirow{3}{*}{
\begin{tikzpicture}[baseline=0mm]
\coordinate(0)at(0,0);\coordinate(u)at(90:1);\coordinate(r)at(-30:1);\coordinate(l)at(210:1);
\draw(0)--node[left]{$1$}(u)node[above]{$p$};\draw(0)--node[below]{$2$}(l)node[below]{$q$};\draw(0)--node[below]{$3$}(r)node[below]{$r$};
\fill(0)circle(0.07);\node at(u){$\times$};\node at(l){$\times$};\node at(r){$\times$};\node at(0,2.3){};
\end{tikzpicture}
} \hspace{1mm}
 &
\begin{tikzpicture}[baseline=-4]\node at(0,0.3){$w(p)=w(q)=w(r)$};\node at(0,-0.3){$=\frac{1}{2}$ or $2$};\end{tikzpicture}
 & $\begin{bmatrix}0&2&-2\\-2&0&2\\2&-2&0\end{bmatrix}$
\begin{tikzpicture}[baseline=-4]\node at(0,0.7){};\node at(0,-0.7){};\end{tikzpicture}%control space
\\\cline{2-3}%
 &
\begin{tikzpicture}[baseline=-4]\node at(0,0.3){$w(p)=w(q)=\frac{1}{2}$};\node at(0,-0.3){$w(r)=2$};\end{tikzpicture}
 & $\begin{bmatrix}0&2&-1\\-2&0&1\\4&-4&0\end{bmatrix}$
\begin{tikzpicture}[baseline=-4]\node at(0,0.7){};\node at(0,-0.7){};\end{tikzpicture}%control space
\\\cline{2-3}%
 &
\begin{tikzpicture}[baseline=-4]\node at(0,0.3){$w(p)=\frac{1}{2}$};\node at(0,-0.3){$w(q)=w(r)=2$};\end{tikzpicture}
 & $\begin{bmatrix}0&1&-1\\-4&0&2\\4&-2&0\end{bmatrix}$
\begin{tikzpicture}[baseline=-4]\node at(0,0.7){};\node at(0,-0.7){};\end{tikzpicture}%control space
\end{tabular}
\caption{Matrices associated with exceptional triangles}
\label{Fig:matrices-exceptional}
\end{figure}
%%%
%%%

We have the cluster algebra $\cA(B_T)$ associated with $T$. We denote by $\bT_T$ the set of tagged triangulations of $\cO^w$ obtained from $T$ by sequences of flips, and by $\bA_T$ the set of tagged arcs of all tagged triangulations in $\bT_T$. Here, Theorem \ref{thm:transitivity} means that $\bT_T$ contains all tagged triangulations of $\cO^w$ except for a weighted orbifold $\cO^w$ with empty boundary and exactly one puncture, in which case $\bT_T$ consists of all tagged triangulations of $\cO^w$ whose tags at the puncture are the same as ones of $T$.

%%%
\begin{theorem}[{\cite[Theorem 9.1]{FeST12a}}]\label{thm:bijection}
Let $T$ be a tagged triangulation of $\cO^w$. There is a bijection
\[
x_T : \bA_T \leftrightarrow \clv B_T.
\]
 Moreover, it induces a bijection
\[
x_T : \bT_T \leftrightarrow \cluster B_T
\]
which sends $T$ to the initial cluster of $\cA(B_T)$ and the flip of $U\in\bT_T$ at $\gamma\in U$ corresponds to the mutation of $x_T(U)$ at $x_T(\gamma)$.
\end{theorem}

Since the map $\frac{1}{w}$ is a weight of $\cO$, there is a new weighted orbifold $\cO^{\frac{1}{w}}$. Tagged arcs $\delta$ and tagged triangulations $T$ of $\cO^w$ are also of $\cO^{\frac{1}{w}}$. To distinguish them, we denote them of $\cO^{\frac{1}{w}}$ by $\delta^{\ast}$ and $T^{\ast}$, respectively.

%%%
\begin{proposition}[{\cite[Lemma 8.6]{FeT17}}]\label{prop:g=-b}
For each $\delta\in\bA_T$, we have
\[
g(x_T(\delta))=-b_{T^{\ast}}(\se(\delta^{\ast})),
\]
where $\se(\delta^{\ast})$ is an elementary laminate of $\cO^{\frac{1}{w}}$.
\end{proposition}

\begin{proof}[Proof of Theorem \ref{thm:dense wo}]
The assertions follows from Theorems \ref{thm:dense lam}, \ref{thm:transitivity}, \ref{thm:bijection}, and \ref{prop:g=-b}.
\end{proof}

Finally, we give an example and finish this paper.

%%%
\begin{example}
Let $\cO^w$ be a monogon with no punctures and exactly two orbifold points $p$ and $q$ such that $w(p)=\frac{1}{2}$ and $w(q)=2$. Notice that $\cO^{\frac{1}{w}}$ is the weighted orbifold in Example \ref{ex:monogon}. Moreover, for the following tagged triangulation $T$ of $\cO^w$, the tagged triangulation $T^{\ast}$ of $\cO^{\frac{1}{w}}$ also appear in Example \ref{ex:monogon} and the associated matrix $B_T$ coincide with $B$ in Example \ref{ex:cluster}:
\[
 T=
\begin{tikzpicture}[baseline=-1mm]
\coordinate(d)at(0,-1);\coordinate(l)at(-0.4,0);\coordinate(r)at(0.4,0);
\draw(l)--node[fill=white,inner sep=2,pos=0.4]{$1$}(d)--node[fill=white,inner sep=2,pos=0.6]{$2$}(r);
\draw(0,0)circle(1);\fill(d)circle(0.07);\node at(l){$\times$};\node at(r){$\times$};\node at(-0.4,0.3){$p$};\node at(0.4,0.3){$q$};
\end{tikzpicture}
\ \text{, where}\hspace{2mm}
B_T=\begin{bmatrix}0&-1\\4&0\end{bmatrix}.
\]
By the examples and Theorem \ref{thm:bijection}, it is easy to check that
\[
\se(x_T^{-1}(x_1)^{\ast})=l_0,\ \se(x_T^{-1}(x_2)^{\ast})=r_0,\ 
\se\Bigl(x_T^{-1}\Bigl(\frac{x_1+y_2}{x_2}\Bigr)^{\ast}\Bigr)=r_1,
\]
\[
\se\Bigl(x_T^{-1}\Bigl(\frac{(x_1+y_2)^4+y_1y_2^4x_2^4}{x_1x_2^4}\Bigr)^{\ast}\Bigr)=l_1,\ 
\se\Bigl(x_T^{-1}\Bigl(\frac{(x_1+y_2)^4+y_1y_2^4x_2^4}{x_1x_2^4}\Bigr)^{\ast}\Bigr)=r_2.
\]
Similarly, all clusters of $\cA(B)$ correspond to laminations $\{l_i,r_i\}$ and $\{l_i,r_{i+1}\}$ of $\cO^{\frac{1}{w}}$ for $i \in \bZ$. By Proposition \ref{prop:g=-b}, the cones $-C(\{l_i,r_i\})$ and $-C(\{l_i,r_{i+1}\})$ give all $g$-vector cones in $\cA(B_T)$. Therefore, Example \ref{ex:dense lam} means that Theorem \ref{thm:dense wo} holds for the matrix in Example \ref{ex:cluster}.
\end{example}

%%%
%%%
%%%
\medskip\noindent{\bf Acknowledgements}.
The author was JSPS Overseas Research Fellow and supported by JSPS KAKENHI Grant Numbers JP20J00410, JP21K13761.

\bibliographystyle{alpha}
\bibliography{bib}

\end{document}